\documentclass[a4paper,12pt]{article}

\newenvironment{proof}{\noindent {\bf Proof:}}{\hfill $\Box$}

\newtheorem{lem}{Lemma}
\newtheorem{prop}{Proposition}

\newtheorem{rem}{Remark}

\newtheorem{ex}{Example}

\usepackage{amsmath}
\usepackage{amssymb}
\usepackage{amsfonts}
\usepackage{graphicx}

\def\e{{\rm e}}
\def\x{\mathbf{x}}

\def\R{\mathbb{R}}

\def\N{\mathbb{N}}

\def\M{\mathbf{M}}

\def\A{\mathbf{A}}

\def\f{\mathbf{f}}
\def\z{\mathbf{z}}
\def\y{\mathbf{y}}

\def\y{\mathbf{y}}

\def\q{\mathbf{q}}

\def\b{\mathcal{B}}

\def\b{\mathbf{b}}
\def\q{\mathbf{q}}
\def\u{\mathbf{u}}

\def\e{\mathbf{e}}

\title
{Mean squared error minimization\\ for inverse moment problems\footnote{D. Henrion acknowledges support by project number 103/10/0628
of the Grant Agency of the Czech Republic. The major part of this work was carried out during
M. Mevissen's stay at LAAS-CNRS, supported by a fellowship within the Postdoctoral Programme
of the German Academic Exchange Service.}}

\begin{document}

\author{Didier Henrion$^{1,2,3}$, Jean B. Lasserre$^{1,2,4}$, Martin Mevissen$^5$}

\footnotetext[1]{CNRS, LAAS, 7 avenue du colonel Roche, F-31400 Toulouse, France.}
\footnotetext[2]{Univ. de Toulouse, LAAS, F-31400 Toulouse, France.}
\footnotetext[3]{Faculty of Electrical Engineering, Czech Technical University in Prague,
Technick\'a 2, CZ-16626 Prague, Czech Republic.}
\footnotetext[4]{Institut de Math\'ematiques de Toulouse, Univ. de Toulouse, UPS, F-31400 Toulouse, France.}
\footnotetext[5]{IBM Research - Ireland, Dublin Technology Campus, Damastown Ind. Park, Mulhuddart, Dublin 15, Ireland.}

\maketitle

\begin{abstract}
We consider the problem of approximating the unknown density $u\in L^2(\Omega,\lambda)$ of a measure $\mu$
on $\Omega\subset\R^n$, absolutely continuous with respect to some given reference measure $\lambda$, from
the only knowledge of finitely many moments of $\mu$. 
Given $d\in\N$ and moments of order $d$, we provide a
polynomial $p_d$ which minimizes the 
mean square error $\int (u-p)^2d\lambda$ over all polynomials $p$ of degree at most $d$.
If there is no additional requirement, $p_d$ is obtained as solution of a linear system.
In addition, if $p_d$ is expressed in the basis of polynomials that
are orthonormal with respect to $\lambda$, its vector of coefficients is just the vector of given moments
and no computation is needed.  Moreover $p_d\to u$ in $L^2(\Omega,\lambda)$ as $d\to\infty$. 
In general nonnegativity of $p_d$ is not guaranteed even though $u$ is nonnegative.
However, with this additional nonnegativity requirement 
one obtains analogous results but computing $p_d\geq0$ that minimizes $\int (u-p)^2d\lambda$
now requires solving an appropriate semidefinite program. We have tested the approach
on some applications arising from the reconstruction of geometrical objects and
the approximation of solutions of nonlinear differential equations.
In all cases our results are significantly better than those obtained with 
the maximum entropy technique for estimating $u$.
\end{abstract}


\maketitle

\section{Introduction}

Estimating the density $u$ of an unknown measure $\mu$ is a well-known problem in statistical analysis, physics or engineering. 
In a statistical context, one is usually given observations in the form of a sample of independent or dependent identically distributed random variables obtained from the unknown measure $\mu$. And so there has been extensive research on estimating the density based on these observations. 
For instance, in one of the most popular approaches, the kernel density estimation \cite{parzen}, the density $u$ is estimated via
a linear combination of kernel functions - each of them being identified with exactly one observation. The crucial step in this method is to choose an appropriate bandwidth for which minimizing the integrated or the mean-integrated squared error between $u$ and its estimate is most common. Another very popular approach uses wavelets \cite{kerkyacharian,donoho, vannucci}, an example of approximating a density by a truncated orthonormal series.
The coefficients in the truncated wavelet expansion are moment estimates derived from the given identically distributed observations.
Again, the approximation accuracy is often measured by the mean-integrated squared error and depends on the number of observations
and the degree of the truncation. 
This approach provides a global density estimate satisfying both good local and periodic approximation properties.
For further details the interested reader is referred to \cite{eloyan, izenman, vannucci} and the many references therein.

In another context - arising in challenging fields such as image recognition, solving nonlinear differential equations, spectral estimation or speech processing - no direct observation is available, but rather finitely many moments of the unknown measure $\mu$ are given. Then the issue 
is to  reconstruct or approximate the density $u$ based on the only knowledge of finitely many moments,
(say up to order $d\in\N$), an {\it inverse} problem from moments. 
A simple method due to \cite{teague} approximates the density $u$ by a polynomial $p$ of degree at most $d$, so that
the moments of the measure $pd\lambda$ matches those of $\mu$, up to order $d$.
However, and in contrast with more sophisticated approaches, the resulting polynomial approximation $p$ is not guaranteed to be a density 
(even though $u$ is) as it 
may takes negative values on the domain of integration. One classical approach to the moment problem is the Pad\'e approximation \cite{pade}
which is based on approximating the measure by a (finite) linear combination of Dirac measures. The Dirac measures and their weights in the decomposition are determined by solving a nonlinear system of equations. In the maximum entropy estimation (another classical approach)
one selects the best approximation of $u$ by maximizing some functional entropy, the most popular 
being the Boltzmann-Shannon entropy. 
In general some type of weak convergence takes place as the degree increases as detailed in \cite{borwein}. Alternatively the norm of the approximate density is chosen as an objective function \cite{bertero, talenti, jones, goodrich}, which allows to show a stronger convergence in norm. In \cite{mead}, maximum entropy and Pad\'e approximates have been compared on some numerical experiments. Finally, piecewise polynomial spline based approaches have also been proposed in \cite{john}.

{\bf Motivation.} Our main motivation to study the (inverse) moment problem arises in the context of the so-called generalized problem of moments (GPM).
The abstract GPM is a infinite-dimensional linear program on some space of Borel measures on $\R^n$ and its applications seem endless, see e.g. \cite{landau,lasserre}
and the many references therein. For instance, to cite a few applications, the GPM framework can be used to help solve
a weak formulation of some ordinary or partial differential equations, as well as some calculus of variations and optimal control problems.
The solution $u$ of the original problem (or an appropriate translate) is interpreted as a density with respect to the Lebesgue measure
$\lambda$ on some domain and one computes (or approximates) finitely many moments of the measure $d\mu:=ud\lambda$ by solving an appropriate 
finite-dimensional optimization problem. But then to recover an approximate solution of the original problem
one has to solve an inverse problem from moments. This approach is particularly attractive when the data of the
original problem consist of polynomials and basic semi-algebraic sets. In this case
one may define a hierarchy (as the number of moments increases) of 
so-called semidefinite programs to compute approximations of increasing quality.

\subsection*{Contribution}

In this paper we consider the following inverse problem from moments: Let $\mu$ be a finite Borel measure 
absolutely continuous with respect to some reference measure $\lambda$ on 
a box $\Omega$ of $\R^n$ and whose density $u$ is assumed to be in $L^2(\Omega,\lambda)$, with no continuity
assumption as in previous works.
The ultimate goal is to compute an {\it approximation} $u_d$ of $u$, based on the only knowledge of 
finitely many moments (say up to order $d$) of $\mu$. In addition, for consistency, it would be highly desirable 
to also obtain some ``convergence" $u_d\to u$ as $d\to\infty$.

(a) Firstly, we approximate the density $u$ by a polynomial $u^*_d$ of degree $d$
which minimizes the mean squared error 
$\int_\Omega (u-p)^2d\lambda$ 
(or equivalently the $L^2(\Omega,\lambda)$-norm $\Vert u-p\Vert_2^2$) over all polynomials $p$
of degree at most $d$.
We show that an unconstrained $L^2$-norm minimizer $u^*_d$ exists, is unique, and coincides with the simple polynomial approximation due to \cite{teague}; it can be determined by solving a system of linear equations. It turns out that $u^*_d$
matches all moments up to degree $d$, and it is even 
easier to compute if it is expressed in the basis of polynomials that are orthonormal with respect to $\lambda$. No inversion is needed
and the coefficients of $u^*_d$ in such a basis are just the given moments.
Moreover we show that $u^*_d\to u$ in $L^2(\Omega,\lambda)$  as $d\to\infty$, which is the best we can hope for
in general since there is no continuity assumption on $u$; in particular
$u^*_{d_k}\to u$ almost-everywhere and almost-uniformly on $\Omega$ for some subsequence $(d_k)$, $k\in\N$. 
Even though both proofs are rather straightforward,
to the best of our knowledge it has not been pointed out before that not only this mean squared error estimate $u^*_d$
is much easier to compute than the corresponding maximum entropy estimate, but 
it also converges to $u$ as $d\to\infty$ in a much stronger sense.
For the {\it univariate case}, in references \cite{provost} and \cite{athanassoulis} the authors address
the problem of approximating a {\it continuous} density on a compact interval by polynomials or kernel density
functions that match a fixed number of moments.
In this case, convergence in supremum norm is obtained when the number of moments increases.
An extension to the noncompact (Stieltjes) case is carried out in \cite{gavriliadis}. 
Notice that
in \cite{provost} it was already observed that the resulting polynomial approximation
also minimizes the mean square error and its coefficients solve
a linear system of equations.
In \cite{bertero, talenti} the minimum-norm solution (and 
not the minimum distance solution) is shown to be unique solution of a system of linear equations. In \cite{jones} the minimal distance solution is considered but it is obtained as the solution of a constrained optimization problem and requires an initial guess for the density estimate.

(b) However, as already mentioned and unlike the maximum entropy estimate, the above 
unconstrained $L^2$-norm minimizer $u^*_d$ may {\it not} be a density as it may take
negative values on $\Omega$. Of course, the nonnegative function $u^*_{d+}:=\max[0,u^*_d]$
also converges to $u$ in $L^2$ but it is not a polynomial anymore. 
So we next propose to obtain a {\it nonnegative polynomial} approximation $u^*_d$ by minimizing the same $L^2$-norm criterion but now under the 
additional constraint that the candidate polynomial approximations 
should be nonnegative on $\Omega$. In principle such a constraint is difficult to handle which probably explains why 
it has been ignored in previous works. Fortunately, if $\Omega$ is a compact basic semi-algebraic set one is able
to enforce this positivity constraint by using Putinar's Positivstellensatz \cite{putinar} which provides a nice positivity certificate for polynomials strictly positive on $\Omega$.
Importantly, the resulting optimization problem is convex and even more, a semidefinite programming (SDP) problem which can be solved 
efficiently by public domain solvers based on interior-point algorithms. Moreover, again as in the unconstrained case, we prove the 
convergence $u^*_d\to u$ in $L^2(\Omega,\lambda)$ as $d\to\infty$ (and so almost-everywhere and almost-uniform convergence on $\Omega$ as well
for some subsequence $(u^*_{d_k})$, $k\in\N$)
which is far stronger than the weak convergence obtained for the maximum entropy estimate. Notice, in \cite{teague, provost} methods for obtaining some non-negative estimates are discussed, however these estimates do not satisfy the same properties in terms of mean-square error minimization and convergence as in the unconstrained case. In the kernel density element method \cite{athanassoulis, gavriliadis} a nonnegative density estimate for the univariate case is obtained by solving a constrained convex quadratic optimization problem. However, requiring each coefficient in the representation to be nonnegative as presented there seems more restrictive than the nonnegative polynomial approximation proposed in this paper.

(c) Our approach is illustrated on some challenging applications. In the first set of problems we are concerned with recovering the shape of geometrical objects whereas in the second set of problems we approximate solutions of nonlinear differential equations. Moreover, we demonstrate the potential of this approach for approximating densities with jump discontinuities, which is harder to achieve than for the smooth, univariate functions discussed in \cite{provost}. The resulting $L^2$-approximations 
clearly outperform the maximum entropy estimates with respect to both running time and pointwise approximation accuracy. Moreover, our approach is able to handle sets $\Omega$ more complicated than a box (as long as the moments of the measure $ud\lambda$ are available) 
as support for the unknown density, whereas such sets are a challenge for computing maximum entropy estimates because
integrals of a combination of polynomials and exponentials of polynomials must be computed repeatedly.

\subsection*{Outline of the paper}

In Section \ref{notations} we introduce the notation and we state the problem to be solved.
In Section \ref{l2} we present our approach to approximate an unknown density $u$ by a polynomial $u^*_d$ of degree at most $d$
via unconstrained and constrained $L^2$-norm minimization, respectively; in both cases we also prove the convergence $u^*_d\to u$ 
in $L^2(\Omega,\lambda)$ (and almost-uniform convergence on $\Omega$ as well for some subsequence) 
as $d$ increases. In Section \ref{secNumEx} we illustrate the approach on a number of examples - most notably from recovering geometric objects and approximating solutions of nonlinear differential equations - and highlight its advantages when compared with the
maximum entropy estimation. Finally, we discuss methods to improve the stability of our approach by considering orthogonal bases for the functions spaces we use to approximate the density. And we discuss the limits of approximating discontinuous functions by smooth functions in connection with the well-known Gibbs effect.

\section{Notation, definitions and preliminaries}\label{notations}

\subsection{Notation and definitions}
Let $\R[\x]$ (resp. $\R[\x]_d$) denote the ring of real polynomials in the variables $\x=(x_1,\ldots,x_n)$ (resp. polynomials of degree at most $d$), whereas $\Sigma[\x]$ (resp. $\Sigma[\x]_d$) denotes 
its subset of sums of squares (SOS) polynomials (resp. SOS of degree at most $2d$).

With $\Omega\subset\R^n$ and a given reference measure $\lambda$ on $\Omega$,
let $L^2(\Omega,\lambda)$ 
be the space of functions on $\Omega$ whose square is $\lambda$-integrable and let
$L^2_+(\Omega,\lambda)\subset L^2(\Omega,\lambda)$ be the convex cone of nonnegative elements.
Let $C(\Omega)$ (resp. $C_+(\Omega)$) be the space of continuous functions 
(resp. continuous nonnegative functions) on $\Omega$.
Let $P(\Omega)$ be the space of polynomials nonnegative on $\Omega$.

For every
$\alpha\in\N^n$ the notation $\x^\alpha$ stands for the monomial $x_1^{\alpha_1}\cdots x_n^{\alpha_n}$ and for every $d\in\N$, let $\N^n_d:=\{\alpha\in\N^{n}:\sum_j\alpha_j\leq d\}$ whose cardinal is $s(d)={n+d\choose d}$.
A polynomial $f\in\R[\x]$ is written 
\[\x\mapsto f(\x)\,=\,\sum_{\alpha\in\N^{n}}\,f_{\alpha}\,\x^\alpha\]
and $f$ can be identified with its vector of coefficients $\f=(f_{\alpha})$ in the canonical basis $(\x^\alpha)$, $\alpha\in\N^n$. 
Denote by ${\mathbb S}^n$ the space of real $n\times n$ symmetric matrices, and by ${\mathbb S}^n_+$ the cone of positive semidefinite
elements of ${\mathbb S}^n$. For any $\A\in{\mathbb S}^n_+$ the notation $\A\succeq0$ stands for positive semidefinite. 
A real sequence $\y=(y_{\alpha})$, $\alpha\in\N^{n}$, has a representing measure if
there exists some finite Borel measure $\mu$ on $\R^{n}$ such that
\[y_{\alpha}\,=\,\int\x^\alpha\,d\mu(\x),\qquad\forall\,\alpha\in\N^{n}.\]

\subsection*{Linear functional}
Given a real sequence $\y=(y_{\alpha})$ define the Riesz linear functional $L_\y:\R[\x]\to\R$ by:
\[f\:(=\sum_{\alpha} f_{\alpha}\x^\alpha)\quad\mapsto L_\y(f)\,=\,\sum_{\alpha}f_{\alpha}\,y_{\alpha},\qquad f\in\R[\x].\]
\subsection*{Moment matrix}
Given $d \in {\mathbb N}$,
the moment matrix of order $d$ associated with a sequence
$\y=(y_{\alpha})$, $\alpha\in\N^{n}$, is the real symmetric matrix $\M_d(\y)$ with rows and columns indexed by $\N^{n}_d$, and whose entry $(\alpha,\beta)$ is $y_{\alpha+\beta}$, for every $\alpha,\beta\in\N^{n}_d$. 
If $\y$ has a representing measure $\mu$ then
$\M_d(\y)\succeq0$ because
\[\langle\f,\M_d(\y)\f\rangle\,=\,\int f^2\,d\mu\,\geq0,\qquad\forall \,\f\,\in\R^{s(d)}.\]

\subsection*{Localizing matrix}
With $\y$ as above and $g\in\R[\x]$ (with $g(\x)=\sum_{\gamma} g_{\gamma}\x^\gamma$), the localizing matrix of order $d$ associated with $\y$ 
and $g$ is the real symmetric matrix $\M_d(g\,\y)$ with rows and columns indexed by $\N^n_d$, and whose entry $(\alpha,\beta)$ is $\sum_{\gamma}g_{\gamma} y_{(\alpha+\beta+\gamma)}$, for every 
$\alpha,\beta\in\N^{n}_d$.
If $\y$ has a representing measure $\mu$ whose support is 
contained in the set $\{\x\,:\,g(\x)\geq0\}$ then
$\M_d(g\,\y)\succeq0$ because
\[\langle\f,\M_d(g\,\y)\f\rangle\,=\,\int f^2\,g\,d\mu\,\geq0,\qquad\forall \,\f\,\in\R^{s(d)}.\]

\subsection{Problem statement}
\label{secL2Mapp}

We consider the following setting. For $\Omega\subset\R^n$ compact, let $\mu$ and $\lambda$ be $\sigma$-finite Borel measures supported on $\Omega$. Assume that the moments of $\lambda$ are known and $\mu$ is absolutely continuous with respect to $\lambda$ ($\mu\ll\lambda$)  with Radon-Nikod\'ym derivative (or density) $u:\, \Omega\rightarrow \R_{+}$, with respect to $\lambda$.  The 
density $u$ is unknown but we know finitely many moments $\y=(y_\alpha)$ of $\mu$, that is,
\begin{equation}
y_{\alpha} := \int_{\Omega} \x^{\alpha} u(\x) d\lambda (\x) = \int_{\Omega} \x^{\alpha} d\mu(\x), \qquad \forall \alpha\in\N^n_d,
\label{givenMoments}
\end{equation}
for some $d\in \N$. 

The issue is to find an estimate $u_d:\, \Omega\rightarrow \R_{+}$ for $u$, such that
\begin{equation}
\int_{\Omega} \x^{\alpha} u_d(\x) d\lambda (\x) = y_{\alpha}, \qquad \forall \alpha\in\N^n_d.
\label{momConditions}
\end{equation}

\subsection{Maximum entropy estimation}
\label{subSecMEE}
We briefly describe the maximum entropy method due to \cite{jaynes1, jaynes2, borwein} as a reference 
for later comparison with the mean squared error approach.

If one chooses the Boltzmann-Shannon entropy $H(u):=-u\log u$, the resulting estimate with maximum-entropy
is an optimal solution of the optimization problem
\[
\max_{u_d} \int_\Omega H(u_d)d\lambda \:\:\mathrm{s.t.}\:\: \int_\Omega \x^\alpha u_d(\x)d\lambda(\x)=y_\alpha,\quad\forall \alpha\in\N^n_d.
\]
It turns out that an optimal solution $u^*_d$ is of the form
\[
\x\mapsto u^*_d(\x):=\displaystyle\exp\left(\sum_{\mid\alpha\mid\leq d} u_{\alpha}\x^{\alpha}\right)
\]
for some vector $\u_d=(u_{\alpha})\in\R^{s(d)}$. Hence, computing an optimal solution $u^*_d$ reduces to solving the finite-dimensional convex 
optimization problem
\begin{equation}
\max_{\u_d\in\R^{s(d)}} \:\left\{\,\langle \y,\u_d\rangle - \int_{\Omega}\exp\left(\sum_{\mid\alpha\mid\leq d}u_{\alpha}\x^{\alpha}\right)\,d\lambda(\x)\right\}
\label{concaveOpt}
\end{equation}
where $\y=(y_\alpha)$ is the given moment information on the unknown density $u$.
If $(u^*_d)$, $d\in\N$, is a sequence of optimal solutions to (\ref{concaveOpt}), then following weak convergence occurs:
\begin{equation}
\label{weak}
\lim_{d\to\infty}\:\int_{\Omega} \psi(\x)u^*_d(\x)\, d\lambda(\x)\,=\, \int_{\Omega}\psi(\x)\,u(\x)\, d\lambda(\x),
\end{equation}
for all bounded measurable functions $\psi:\Omega\rightarrow\R$ continuous almost everywhere. 
For more details the interested reader is referred to \cite{borwein}.

Since the estimate $u^*_d$ is an exponential of a polynomial, it is guaranteed to be nonnegative on $\Omega$ and so
it is a density. However, even though the problem is convex it remains hard to solve because 
in first or second-order optimization algorithms, computing the gradient or Hessian at a current iterate $\u_d=(u_\alpha)$
requires evaluating integrals of the form
\[\int_\Omega \x^\alpha\displaystyle\exp\left(\sum_{\mid\alpha\mid\leq d} u_{\alpha}\x^{\alpha}\right)\,d\lambda(\x),\quad\alpha\in\N^n_d\]
which is a difficult task in general, except perhaps in small dimension $n=1$ or $2$.

\section{The mean squared error approach}\label{l2}

In this section we assume that the unknown density $u$ is an element
of $L^2(\Omega,\lambda)$, and we introduce our mean squared error, or $L^2$-norm approach, for density approximation.

\subsection{Density approximation as an unconstrained problem}
\label{unconSubSec}

We now restrict $u_d$ to be a polynomial, i.e.
\[
\x\mapsto\quad u_d(\x)\,:=\,\sum_{\mid\alpha\mid\leq d} u_{\alpha}\x^{\alpha}
\]
for some vector of coefficients $\u=(u_\alpha)\in\R^{s(d)}$.
We first show how to obtain a polynomial estimate $u^*_d\in\R [\x]_d$ of $u$ satisfying (\ref{momConditions})
by solving an unconstrained optimization problem.

Let $\z$ denote the sequence of moments of $\lambda$ on $\Omega$, i.e., $\z=(z_{\alpha})$, $\alpha\in\N^n$, with
\[
z_{\alpha} = \int_{\Omega} \x^{\alpha} d\lambda,\qquad \forall\,\alpha\in\N^n,
\]
and let $\M_d(\z)$ denote the moment matrix of order $d$ of $\lambda$.
This matrix is easily computed since the moments of $\lambda$ are known.

Consider the unconstrained optimization problem 
\begin{equation}
\displaystyle\min_{u_d\in\R[\x]_d} \: \displaystyle\Vert u-u_d\Vert_2^2 \:\left(= \int_{\Omega} (u - u_d)^2 d\lambda\right).
\label{unconOpt}
\end{equation}

\begin{prop}
\label{L2minProp}
Let $\Omega\subset\R^n$ have a nonempty interior and let $\lambda(O)>0$ for some open set $O\subset\Omega$.
A polynomial $u^*_d\in\R[\x]_d$ is  an optimal solution of problem (\ref{unconOpt}) if and only if
its vector of coefficients $\u^*_d\in\R^{s(d)}$ is an optimal solution of the unconstrained quadratic optimization problem
\begin{equation}
\label{unconQP}
\displaystyle \min_{\u_d\in\R^{s(d)}}\,\{\: \u^T_d \, \M_d(\z) \, \u_d - 2\,\u^T_d \y\,\}.
\end{equation}
Then $\u^*_d:=\M_d^{-1}(\z)\y$ is the unique solution of (\ref{unconQP}), and $u^*_d\in\R[\x]_d$ satisfies:
\begin{equation}
\label{eq-moments}
\int_\Omega \x^\alpha\,u^*_d\,d\lambda\,=\,y_\alpha\,=\,\int_\Omega \x^\alpha\,u\,d\lambda,\qquad\forall\,\alpha\in\N^n_d.\end{equation}
\end{prop}
\begin{proof}
Observe that for every $u_d\in\R[\x]_d$ with vector of coefficients $\u_d\in\R^{s(d)}$,
\begin{eqnarray*}
\int_{\Omega} \left(u -u_d\right)^2 d\lambda
 &=& \displaystyle\int_{\Omega} u^2_d\, d\lambda -2\,\int_{\Omega} \underbrace{u_d\,u\,d\lambda}_{u_dd\mu} + \int_{\Omega} u^2 d\lambda\\
 &=& \u^T_d\M_d(\z)\u_d - 2\,\u^T_d\y + \int_{\Omega} u^2\,d\lambda.
\end{eqnarray*}
The third term on the right handside being constant, it does not affect the optimization and can be ignored.
Thus, the first claim follows.

The second claims follows from the well-known optimality conditions for unconstrained, convex quadratic programs
and the fact that $\M_d(\z)$ is nonsingular because $\M_d(\z)\succ 0$ for all $d\in\N$. Indeed,
if $\q^T\M_d(\z)\q=0$ for some $0\neq\q\in\R^{s(d)}$ then necessarily the polynomial $q\in\R[\x]_d$
with coefficient vector $\q$ vanishes on the open set $O$, which implies that $q=0$,
in contradiction with $\q\neq0$.

Finally, let $\e_\alpha\in\R^{s(d)}$ be the vector of coefficients associated with the monomial $\x^\alpha$, $\alpha\in\N^n_d$.
from $\M_d(\z)\u^*_d=\y$ we deduce 
\[y_\alpha\,=\,\e_\alpha^T\M_d(\z)\u^*_d\,=\,\int_\Omega \x^\alpha u^*_d\,d\lambda\,\]
which is the desired result.
\end{proof}

Thus the polynomial $u^*_d\in\R[\x]_d$ minimizing the $L^2$-norm distance to $u$
coincides with the polynomial approximation due to \cite{teague} defined to be 
a polynomial which satisfies all conditions (\ref{momConditions}). 
Note that this is not the case anymore if one uses an $L^p$-norm distance with $p>2$.

Next, we obtain the following convergence result for the sequence of minimizers of problem (\ref{unconOpt}), $d\in\N$.

\begin{prop}
\label{propConv}
Let $\Omega$ be compact with nonempty interior and let $\lambda$ be finite with $\lambda(O)>0$
for some open set $O\subset\Omega$. Let $(u^*_d)$, $d\in\N$, be  the sequence of minimizers of problem (\ref{unconOpt}). Then
$\Vert u-u^*_d\Vert_2 \rightarrow 0$ as $d\to\infty$. In particular there is a subsequence $(d_k)$, $k\in\N$, such that
$u^*_{d_k}\to u$, $\lambda$-almost everywhere and $\lambda$-almost uniformly on $\Omega$, as $k\to\infty$.
\end{prop}
\begin{proof}
Since $\Omega$ is compact, $\R [\x]$ is dense in $L^2(\Omega,\lambda)$. Hence as $u\in L^2(\Omega,\lambda)$ there exists a sequence 
$(v_k)\subset\R [\x]$, $k\in\N$, with $\Vert u-v_k\Vert_2\rightarrow 0$ as $k\to\infty$. Observe that
if $d_k=\deg\,v_k$ then as $u^*_{d_k}$ solves problem (\ref{unconOpt}), it holds that 
$\Vert u-v_k\Vert_2\,\geq\,\Vert u-u^*_{d_k}\Vert_2$ for all $k$,
which combined with $\Vert u-v_k\Vert_2\rightarrow 0$ yields the desired result. The last statement follows from 
\cite[Theorem 2.5.2 and 2.5.3]{ash}.
\end{proof}

Note that computing the $L^2$ norm minimizer $u^*_d$ is equivalent to solving a system of linear equations, whereas computing 
the maximum entropy estimate requires solving the potentially hard convex optimization problem (\ref{concaveOpt}).
Moreover, the $L^2$-convergence $\Vert u^*_d-u\Vert_2\to0$ in Proposition \ref{propConv} (and so almost-everywhere and 
almost-uniform convergence for a subsequence) 
is much stronger than the weak convergence
(\ref{weak}). On the other hand, unlike the maximum entropy estimate, the $L^2$-approximation $u^*_d\in\R[\x]_d$ is not
guaranteed to be nonnegative on $\Omega$, hence it is not necessarily a density.
Methods to overcome this shortcoming are discussed in \S \ref{conSubSec}.

\begin{rem}
\label{frameRmk}
{\rm In general the support $K:=\text{supp}\,\mu$ of $\mu$ may be a strict subset of $\Omega=\text{supp}\,\lambda$. In the case where
$K$ is not known or its geometry is complicated, one chooses a set $\Omega\supset K$ with a simple geometry so that moments of $\lambda$
are computed easily. As demonstrated on numerical examples in Section \ref{secNumEx}, choosing an enclosing frame 
$\Omega$ as tight as possible is crucial for reducing the pointwise approximation error of $u^*_d$ when the degree $d$ is fixed. In the maximum entropy
method of \S \ref{subSecMEE} no enclosing frame $\Omega\supset K$ is chosen. 
In the $L^2$-approach, choosing $\Omega\supset K$ is a degree of freedom 
that sometimes can be exploited for a fine tuning of the approximation accuracy.
}\end{rem}

\begin{rem}
{\rm From the beginning we have considered a setting where an exact truncated moment vector $\y$ of the 
unknown density $u\geq 0$ is given. However, usually one only has an approximate moment vector $\tilde{\y}$ 
and in fact, it may even happen that $\tilde{\y}$ is not the moment vector of a (nonnegative) measure. In the latter case, the maximum of the convex problem (\ref{concaveOpt}) is unbounded, whereas the $L^2$-norm approach  always yields a polynomial estimate $u^*_d\in\R[\x]_d$.
}\end{rem}

\begin{rem}
{\rm
If $\tilde{\y}$ is a slightly perturbed version of $\y$, the resulting numerical error in $u_d$ and side effects caused by ill conditioning may be reduced by considering the regularized problem
\begin{equation}
\label{unconOptRegular}
\displaystyle\min_{u_d}\: \int_{\Omega} (u-u_d)^2 d\lambda + \epsilon\Vert\u_d\Vert^2
\end{equation}
where $\Vert\u_d\Vert^2$ is the Euclidean norm of the coefficient vector $\u_d \in \R^{s(d)}$ of $u_d\in\R[\x]_d$, and $\epsilon>0$ (fixed)
is a regularization parameter approximately of the same order as the noise in $\tilde{\y}$. The coefficient vector 
of an optimal solution $u^*_d(\epsilon)$ of (\ref{unconOptRegular}) is then given by
\begin{equation}
\u^*_d(\epsilon) = \left( \M_d + \epsilon I\right)^{-1}\tilde{\y}.
\end{equation}
The effect of small perturbations in $\y$ on the pointwise approximation accuracy of $u^*_d$ for $u$ is demonstrated on some numerical examples in 
Section \ref{secNumEx}. However, a more detailed analysis of the sensitivity of our approach for noise or errors in the given moment information is beyond the scope of this paper.
}\end{rem}

\subsection{Density approximation as a constrained optimization problem}
\label{conSubSec}

As we have just mentioned, the minimizer $u^*_d$ of problem (\ref{unconOpt}) is not guaranteed to yield a nonnegative approximation 
even if $u\geq 0$ on $\Omega$. As we next see, the function $\x\mapsto u^*_{d+}(\x):=\max[0,u^*_d(\x)]$
also converges to $u$ for the $L^2$-norm but
\begin{itemize}
\item it does not 
satisfy the moments constraints (i.e., does not match the moments of $d\mu=ud\lambda$ up to order $d$);
\item it is not a polynomial anymore (but a piecewise polynomial).
\end{itemize}
In the sequel, we address the second point by approximating the density with a polynomial nonnegative on $\Omega$,
which for practical purposes, is easier to manipulate than a piecewise polynomial.
We do not address explicitly the first point, which is not as crucial in our opinion. Note however that at the price
of increasing its degree and adding linear constraints, the resulting polynomial approximation may match an (a priori)
fixed number of moments.

Adding a polynomial nonnegativity constraint to problem (\ref{unconOpt})
yields the constrained optimization problem:
\begin{equation}
\displaystyle\min_{u_d\in\R[\x]_d} \:\{ \displaystyle\Vert u-u_d\Vert_2^2 \::\: u_d \geq 0 \:\text{ on }\Omega\}
\label{nonnegOpt}
\end{equation}
which, despite convexity, is untractable in general. We consider two alternative optimization problems to enforce nonnegativity of the 
approximation; the first one considers necessary conditions of positivity whereas the second one considers sufficient conditions for positivity.

\subsection*{Necessary conditions of positivity}
Consider the optimization problem:
\begin{equation}
\displaystyle\min_{u_d\in\R[\x]_d}\:\{ \displaystyle\Vert u-u_d\Vert_2^2 \::\: \M_d(u_d\,\z) \succeq 0\,\},
\label{noliftOpt}
\end{equation}
where $\z=(z_{\alpha})$ is the moment sequence of $\lambda$ and $\M_d(u_d\,\z)$ is the localizing matrix associated with
$u_d$ and $\z$. Observe that problem (\ref{noliftOpt}) is a convex optimization problem because the objective function
is convex quadratic and the feasible set is defined by linear matrix inequalities (LMIs).

The rationale behind the semidefiniteness constraint $\M_d(u_d\,\z)\succeq0$ in (\ref{noliftOpt}) follows from
a result in \cite{newlook} which states that if ${\rm supp}\,\lambda=\Omega$ and $\M_d(u_d\,\z)\geq0$ for all $d$ then
$u_d\geq0$ on $\Omega$.

\begin{lem}
\label{denseLemma}
If $\Omega$ is compact then $P(\Omega)$ is dense in $L^2(\Omega,\lambda)_+$ with respect to the $L^2$-norm.
\end{lem}
\begin{proof}
As $\Omega$ is compact the polynomials are dense in $L^2(\Omega,\lambda)$. Hence
there exists a sequence $(u_d)\subset\R[\x]$, $d\in\N$, such that $\Vert u-u_d\Vert_2\to0$ as $d\to\infty$.
But then the sequence $(u_d^+)$, $d\in\N$, with $u_d^+(\x):=\max[0,u_d(\x)]$ for all $\x\in\Omega$, also converges
for the $L^2$-norm. Indeed,
\begin{eqnarray*}
\int_\Omega(u-u_d)^2d\lambda&=&\displaystyle\int_{\Omega\cap\{\x: u_d(\x)<0\}}(u-u_d)^2d\lambda+\displaystyle\int_{\Omega\cap\{\x: u_d(\x)\geq0\}}(u-u_d)^2d\lambda\\
&\geq&\displaystyle\int_{\Omega\cap\{\x: u_d(\x)<0\}}u^2d\lambda+\displaystyle\int_{\Omega\cap\{\x: u_d(\x)\geq0\}}(u-u_d)^2d\lambda\\
&=&\int_\Omega(u-u_d^+)^2d\lambda=\Vert u-u_d^+\Vert_2^2.
\end{eqnarray*}
So let $k\in\N$ and $d_k\in\N$ be such that $\Vert u-u_{d_k}\Vert_2<k^{-1}$ and so $\Vert u-u_{d_k}^+\Vert_2<k^{-1}$.
As $u_{d_k}^+$ is continuous and $\Omega$ is compact, by the Stone-Weierstrass theorem there exists a sequence
$(v_{d_k\ell})\subset\R[\x]$, $\ell\in\N$, that converges to $u_{d_k}^+$ for the supremum norm.
Hence $\sup_{\x\in\Omega}\vert u_{d_k}^+-v_{d_k\ell}\vert<k^{-1}$ for all $\ell\geq\ell_{k}$ (for some $\ell_k$).
Therefore, the polynomial $w_{d_k}:=v_{d_k\ell_k}+k^{-1}$ is positive on $\Omega$ and 
\begin{eqnarray*}
\Vert u-w_{d_k\ell_k}\Vert_2&\leq&\underbrace{\Vert u-u_{d_k}^+\Vert_2}_{<k^{-1}}+\Vert u_{d_k}^+-w_{d_k\ell_k}\Vert_2\\
&\leq&k^{-1}+\underbrace{\Vert u_{d_k}^+-v_{d_k\ell_k}\Vert_2}_{<k^{-1}\lambda(\Omega)^{1/2}}+\underbrace{\Vert v_{d_k\ell_k}-w_{d_k\ell_k}\Vert_2}
_{=k^{-1}\lambda(\Omega)^{1/2}}\\
&\leq&k^{-1}+2k^{-1}\lambda(\Omega)^{1/2}.
\end{eqnarray*}
Therefore we have found a sequence $(w_{d_k\ell_k})\subset P(\Omega)$, $k\in\N$, such that
$\Vert u-w_{d_k\ell_k}\Vert_2\to0$ as $k\to\infty$.
\end{proof}

\begin{prop}
\label{propConvnolift}
Let $\Omega$ be compact with nonempty interior. Then
problem (\ref{noliftOpt}) has an optimal solution $u^*_d$ for every $d\in\N$,
and $\Vert u-u^*_d\Vert_2 \to0$ as $d\to\infty$.
\end{prop}
\begin{proof}
Fix $d$ and
consider a minimizing sequence $(u_\ell)\subset\R[\x]_d$ with $\Vert u-u_\ell\Vert_2^2$ monotonically decreasing and
converging to a given value as $\ell\to\infty$.
We have $\Vert u_\ell\Vert_2\leq \Vert u\Vert_2+\Vert u-u_\ell\Vert_2\leq \Vert u\Vert_2+\Vert u-u_0\Vert_2$ for all $\ell \in \N$.
Therefore as $\Vert \cdot\Vert_2$ defines a norm on the finite dimensional space $\R[\x]_d$ ($\Omega$ has nonempty interior)
the whole sequence $(u_\ell)$ is contained in the ball $\{v\,:\,\Vert v\Vert_2\leq \Vert u\Vert_2+\Vert u-u_0\Vert_2\}$.
As the feasible set is closed, problem (\ref{noliftOpt}) has an optimal solution.
 
Let $(u^*_d)\subset\R[\x]$ be a sequence of optimal solutions of problem (\ref{noliftOpt}). By Lemma \ref{denseLemma},
$P(\Omega)$ is dense in $L^2(\Omega)_+$. Thus there exists a sequence $(v_k)\subset\R [\x]_{\geq 0}$,
$k\in\N$, with $\Vert u-v_k\Vert_2\to 0$ as $k\to\infty$. As $v_k\geq0$ on $\Omega$
then necessarily $\M_d(v_k\,\z)\succeq0$ for all $d$ and all $k$. In particular $\M_{d}(v_{k_d}\,\z)\succeq0$ 
where $k_d:=\max\{k\,:\,{\rm deg}\,v_k\leq d\}$. 
Therefore  $v_{k_d}\in\R[\x]_d$ is a feasible solution of problem (\ref{noliftOpt}) which yields
$\Vert u-v_{k_d}\Vert_2^2\geq\Vert u-u^*_d\Vert_2^2$ for all $d$. Combining with $\Vert u-v_{k_d}\Vert_2\to 0$ yields the desired result.
\end{proof}

\subsection*{Via sufficient conditions of positivity}

Let $\b_d(\x):=(1,x_1,\ldots,x_n,x_1^2, x_1x_2,\ldots,x_n^d)^T$ denote the standard monomial basis of $\R [\x]_d$.
Let $\Omega$ be a basic compact semi-algebraic set
defined by $\Omega=\{\x\in\R^n\mid g_j(\x)\geq 0,\:j=1,\ldots ,m\}$ for some polynomials
$g_j\in\R[\x]$, $j=1,\ldots,m$ with $d_j=\lceil(\deg\,g_j)/2\rceil$. Then, with $d\geq \max_jd_j$, consider the optimization problem:
\begin{equation}
\begin{array}{cl}
\displaystyle\min_{u_d\in\R[\x]_d} & \displaystyle\Vert u-u_d\Vert_2^2\\
\text{s.t.} & u_d(\x) = \b_d(\x)^T \A_0\b_d(\x) + \sum_{j=1}^m \b_{d-d_j}(\x)^T \A_j\b_{d-d_j}(\x)g_j(\x), \\
& \A_0 \in {\mathbb S}_{+}^{s(d)},\:\A_j \in {\mathbb S}_{+}^{s(d-d_j)}, \quad j=1,\ldots,m.
\end{array}
\label{putinarOpt}
\end{equation}
Since the equality constraints are linear in $u_d$ and the entries of $\A_j$, $j=0,\ldots,m$, the feasible set
of (\ref{putinarOpt}) is a convex LMI set. Moreover, the objective function is convex quadratic.

Note that whereas the semidefinite constraint $\M_d(u_d\,\z)\succeq0$ in (\ref{noliftOpt}) was a relaxation of 
the nonnegativity constraint $u_d\geq 0$ on $\Omega$, in (\ref{putinarOpt}) a feasible solution $u_d$
is necessarily nonnegative on $\Omega$ because the LMI constraint on $u_d$
is a Putinar certificate of positivity on $\Omega$.
However, in problem (\ref{putinarOpt}) we need to introduce $m+1$ auxiliary matrix variables $\A_j$, whereas (\ref{noliftOpt}) is an optimization problem in the original coefficient vector $\u_d$ and does not require such a lifting. Thus, problem (\ref{noliftOpt}) is computationally easier to handle than problem (\ref{putinarOpt}), although both are convex SDPs, which are substantially harder to solve than problem (\ref{unconQP}). 
  
\begin{prop}
\label{propConvLift}
Let $\Omega$ be compact with nonempty interior. For every $d\geq\max_jd_j$, problem (\ref{putinarOpt})
has an optimal solution $u^*_d$ and $\Vert u-u^*_{d}\Vert^2_2\to 0$ as $d\to\infty$.
\end{prop}
\begin{proof}
With $d\geq\max_jd_j$ fixed, consider a minimizing sequence $(u^*_d)\subset\R[\x]_d$ for problem (\ref{putinarOpt}).
As in the proof of Proposition \ref{propConvnolift} one has $\Vert u^*_d\Vert_1\leq \Vert u^*_d\Vert_2\leq \Vert u\Vert_2+\Vert u-u^*_0\Vert_2$.
As $u^*_d$ is feasible,
\[u^*_d(\x) = \b_d(\x)^T \A^d_0\b_d(\x) + \sum_{j=1}^m \b_{d-d_j}(\x)^T \A^d_j\b_{d-d_j}(\x)g_i(\x),\]
for some real symmetric matrices $\A^d_j\succeq0$, $j=0,\ldots,m$.
Rewriting this as
\[u^*_d(\x)\,=\,\langle\A^d_0,\b_d(\x)\b_d(\x)^T\rangle+\sum_{j=1}^m \langle \A^d_j,\b_{d-d_j}(\x)\b_{d-d_j}(\x)^T\rangle,\]
and integrating with respect to $\lambda$ yields
\[\langle\A^d_0,\M_d(\z)\rangle+\sum_{j=1}^m \langle \A^d_j,\M_{d-d_j}(g_j\,\z)\rangle\,=\,\int_\Omega u^*_d\,d\lambda\,\leq\,
\Vert u^*_d\Vert_1\,\leq\,a,\]
with $a:=\Vert u^*_d\Vert_1\leq\Vert u\Vert_2+\Vert u-u^*_0\Vert_2$.
Hence, for every $d$,
\[\langle \A_0,\M_d(\z)\rangle\,\leq\,a\quad\mbox{ and }\quad\langle \A^d_j,\M_{d-d_j}(g_j\,\z)\rangle\,\leq\,a,\quad j=1,\ldots,m.\]
As $\M_d(\z)\succ0$, $\M_{d-d_j}(g_j\,\z)\succ0$, and $\A^d_j\succeq0$, $j=1,\ldots,m$, we conclude that
all matrices $\A^d_j$ are bounded. Therefore the minimizing sequence $(u^*_d,(\A^d_j))$ belongs to a
closed bounded  set and as the mapping $v\mapsto \Vert u-v\Vert_2^2$ is continuous, an optimal solution exists.

From Lemma \ref{denseLemma} there exists $(v_k)\subset \R[\x]_{\geq0}$, $k\in\N$, such that 
$\Vert u-v_k\Vert^2_2\to 0$ as $k\to\infty$. Using properties of norms,
$\Vert u-v_k\Vert_2-k^{-1}\,\leq\,\Vert u-(v_k+k^{-1})\Vert_2\,\leq\,\Vert u-v_k\Vert_2+k^{-1}$,
and so $\Vert u-(v_k+k^{-1})\Vert_2\to 0$ as $k\to\infty$.
Moreover, as $v_k+k^{-1}$ is strictly positive on $\Omega$, by Putinar's Positivstellensatz \cite{putinar}, there exists $d_k$ such that 
\[v_k(\x)= \b_{d_k}(\x)^T \A_0\b_{d_k}(\x) + \sum_{j=1}^m \b_{d_k-d_j}(\x)^T \A_j\b_{d_k-d_j}(\x)g_j(\x), \quad\forall \x,\]
for some real matrices $\A_j \succeq 0$, $j=0,\ldots,m$. So letting 
$d^+_k=\max[{\rm deg}\,v_k,d_k]$, the polynomial $v_k+k^{-1}$ is a feasible solution of problem (\ref{putinarOpt}) whenever $d\geq d^+_k$
and with value $\Vert u-(v_k+k^{-1})\Vert^2_2\,\geq\,\theta_{d^+_k}$. Hence 
$\Vert u-v_{d^+_k}\Vert^2_2\to 0$ as $k\to\infty$, and by monotonicity of the sequence 
$(\Vert u-u^*_d\Vert_2^2)$, $d\in\N$, the result follows.
\end{proof}

\begin{rem}
{\rm Since $\Omega$ is compact, Proposition \ref{propConvnolift} and \ref{propConvLift} imply 
that minimizers of the two constrained $L^2$ norm minimization problems 
(\ref{noliftOpt}) and (\ref{putinarOpt}) converge almost uniformly to the 
unknown density $u$ as in the unconstrained case.
}\end{rem}

\begin{rem}
\label{rmkDom}
{\rm Our approach can handle quite general sets $\Omega$ and $K$ as support and frame for the unknown measure in the unconstrained and constrained cases,
the  only restriction being that (i) $\Omega$ and $K$ are basic compact semi-algebraic sets, and (ii)
one can compute all moments of $\lambda$ on $\Omega$.
In contrast, to solve problem (\ref{concaveOpt}) by local minimization algorithms
using gradient and possibly Hessian information, integrals of the type
$$ \int_{\Omega}\x^{\alpha}\exp\left(\sum_{\beta}u_{\beta}\x^{\beta}\right)d\lambda(\x)$$
must be evaluated. Such evaluations may be difficult as soon
as $n\geq3$. In particular in higher dimensions, cubature formulas for approximating such integrals are difficult to obtain if $\Omega$ is 
not a box or a simplex.
}\end{rem}

\section{Numerical experiments}
\label{secNumEx}

In this section, we demonstrate the potential of our method on a range of examples. We measure the approximation error between a density $u$ and its estimate
$u_d$ by the average error
\[
\bar{\epsilon}_d:=\int_{\Omega} |u(x)-u_d(x)| d\lambda(x)
\]
and the maximum pointwise error
\[
\hat{\epsilon}_d:=\max_{x\in\Omega} |u(x)-u_d(x)|.
\]
In some examples we also consider $\bar{\epsilon}^o_d$ and $\hat{\epsilon}^o_d$, the respective errors on particular segments of the interior of $\Omega$.
We compare the performance of our approach to the maximum entropy estimation from \S \ref{subSecMEE}. Both methods are encoded in Matlab. We implemented the
mean squared error ($L^2$) minimization approach for the standard monomial basis, which results in solving a linear system in the unconstrained case. In the constrained case we apply SeDuMi to solve the resulting SDP problem. In both cases, the numerical stability can be improved by using an orthogonal basis (such as Legendre or Chebychev polynomials) for $\R[x]$. In order to solve the unconstrained, concave optimization problem (\ref{concaveOpt}), we apply the Matlab Optimization Toolbox command {\tt fminunc} as a black-box solver. The observed performance of the maximum entropy estimation (MEE)
method may be improved when applying more specialized software.

\begin{figure}[h!]
\begin{center}
\includegraphics[width=0.45\textwidth]{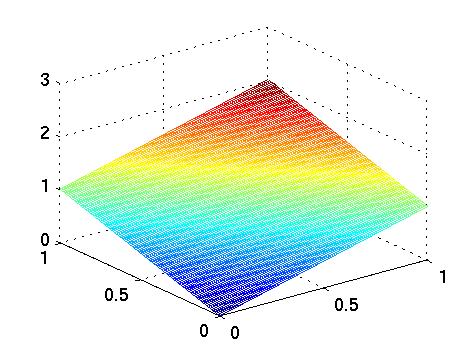} 
\includegraphics[width=0.45\textwidth]{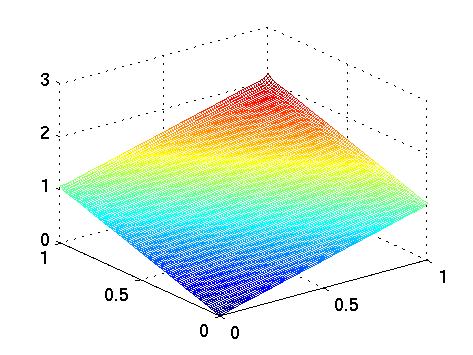} 
\caption{Degree 10 approximation of the density from exact moment vector (left) vs. perturbed moment vector (right).\label{mnatsaFig}}
\end{center}
\end{figure}

\begin{ex}
\label{mnatsaEx}
First we consider the problem of retrieving the density $u(\x)=x_1+x_2$
on $[0, 1]^2$ given its moments $y_{\alpha}=\int_{\Omega}\x^{\alpha}u(x)d\,\x=\frac{1}{(\alpha_1+1)(\alpha_2+2)} + \frac{1}{(\alpha_1+2)(\alpha_2+1)}$, which was considered as a test case in \cite{mnatsa}. This example is a priori very favorable for our technique since the desired $u$ is a polynomial itself. When solving problem (\ref{unconOpt}) for $d\in\{3,5,10\}$ we do obtain the correct solution $\u^*_d=(0,1,1,0,\ldots)^T$ in less than $0.1$ secs. Thus, the pointwise approximation is much better than the one achieved in \cite{mnatsa}. Moreover, $u^*_d \geq 0$ without adding LMI constraints.\\
The question arises of how the polynomial approximation behaves if the moment vector contains some noise, i.e. if it does not exactly coincide with the moment vector of the desired $u$. We solve problem (\ref{unconOpt}) for $\tilde{\y}:=\y+\epsilon$ where the maximal relative componentwise perturbation $\epsilon$ between $\y$ and $\tilde{\y}$ is less than $3\%$. The pointwise error between $u^*_{10}(\y)$ and the solution for the perturbed moment vector $u^*_{10}(\tilde{\y})$ is sufficiently small, as pictured in Figure \ref{mnatsaFig}.
\end{ex}

\begin{table}
\begin{center}
\begin{tabular}{|c|c|c|c|}
\hline $d$ & problem & $\bar{\epsilon}_d$ & $\hat{\epsilon}_d$ \\
\hline $20$ & (\ref{unconOpt}) & $0.0031$ & $0.0296$ \\
\hline $30$ & (\ref{unconOpt}) & $0.0022$ & $0.0265$ \\
\hline $50$ & (\ref{unconOpt}) & $0.0021$ & $0.0251$ \\
\hline
\end{tabular}
\caption{Estimating the density $u(x)=|x|$ from an exact moment vector.\label{absFunResultsExact}}
\end{center}
\end{table}

\begin{table}
\begin{center}
\begin{tabular}{|c|c|c|c|c|c|}
\hline $d$ & problem & $\bar{\epsilon}_d$ & $\hat{\epsilon}_d$ & $\bar{\epsilon}^o_d$ & $\hat{\epsilon}^o_d$ \\ 
\hline $10$ & (\ref{unconOpt}) & $0.0252$ & $0.2810$ & $0.0207$ & $0.1358$ \\
\hline $20$ & (\ref{unconOpt}) & $0.0244$ & $0.7934$ & $0.0142$ & $0.0952$ \\
\hline $30$ & (\ref{unconOpt}) & $0.0237$ & $1.0956$ & $0.0112$ & $0.1024$ \\
 $30$ & (\ref{putinarOpt}) & $0.0176$ & $0.4705$ & $0.0106$ & $0.0937$ \\
\hline $50$ & (\ref{unconOpt}) & $0.0236$ & $1.4591$ & $0.0088$ & $0.1023$ \\
$50$ & (\ref{putinarOpt}) & $0.0206$ & $0.6632$ & $0.0118$ & $0.0979$ \\
\hline
\end{tabular}
\caption{Estimating the density $u(x)=|x|$ from a perturbed moment vector.\label{absFunResultsNoise}}
\end{center}
\end{table}

\begin{figure}[ht]
\begin{center}
\includegraphics[width=0.45\textwidth]{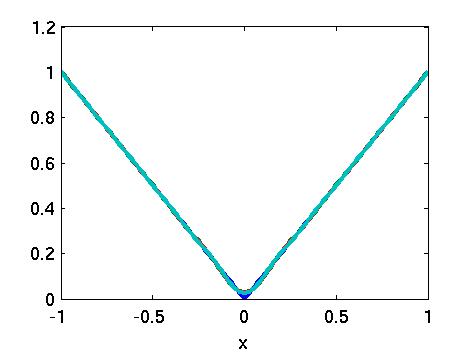} 
\includegraphics[width=0.45\textwidth]{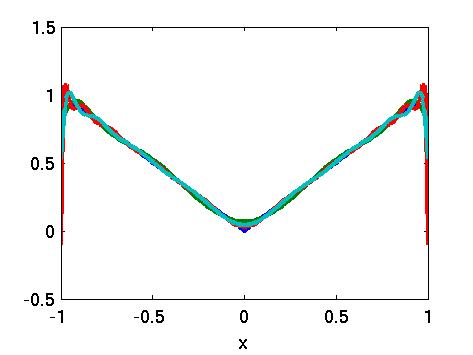} 
\caption{Mean squared error minimizers for $u(x)=|x|$ from exact (left) and perturbed (right) moments.
Blue: $u$, green: $u_{20}$, red: $u_{30}$, cyan: $u_{50}$.\label{absFunFig}}
\end{center}
\end{figure}

\begin{ex}
\label{absFunEx}
Next we consider recovering the function $u:[-1,1]\to\R$ with $u(x)=|x|$ as a first example of a nondifferentiable function. 
In a first step we solve problem (\ref{unconOpt}) for the exact moment vector with entries $y_k:=\int_{-1}^1 |x|x^k dx=\frac{1+(-1)^k}{k+2}$, $k=0,\ldots,d$ 
corresponding to the density $u$. The resulting estimates $u^*_d$ for $d\in\{20,30,50\}$ provide a highly accurate pointwise approximation of $u$ on the entire domain, as reported in Table \ref{absFunResultsExact} and pictured in Figure \ref{absFunFig} (left).\\
In a second step we consider a perturbed moment vector $\tilde{\y}$ as input.
When solving problem (\ref{unconOpt}), we observe that both errors $\hat{\epsilon}_d$, $\bar{\epsilon}_d$ on the interior of $[-1,1]$ -- in particular at the nondifferentiable point $x=0$ -- decrease for increasing $d$, whereas the pointwise error increases at the boundary of the domain. 
Although providing good approximations for $u$ on the entire interior of $[-1,1]$, the estimates $u^*_{30}$ and $u^*_{50}$ take negative values at the boundary $\{-1,1\}$. This is circumvented by solving problem (\ref{putinarOpt}). The new estimates are globally nonnegative while their approximation accuracy
is only slightly worse than in the unconstrained case, as reported in Table \ref{absFunResultsNoise} and pictured in Figure \ref{absFunFig} (right). 
\end{ex}

\begin{figure}[ht]
\begin{center}
\includegraphics[width=0.45\textwidth]{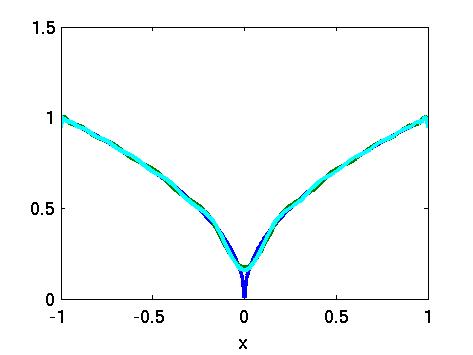} 
\includegraphics[width=0.45\textwidth]{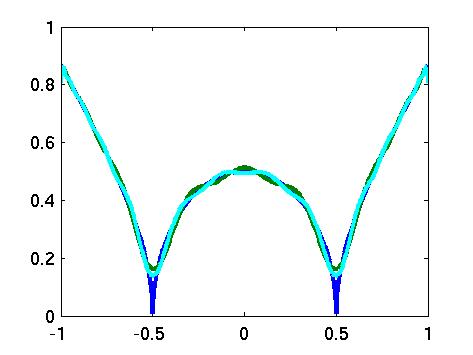}
\caption{Mean squared error minimizers for $u_1$ (left) and $u_2$ (right) from exact moments.
Blue: $u$, green: $u_{20}$, cyan: $u_{50}$.\label{notLocLipFig}}
\end{center}
\end{figure}

\begin{ex}
\label{nonLocLipEx}
Consider the functions $u_1,\, u_2:[-1,1]\rightarrow\R, \; u_1(x)=|x|^{\frac{1}{2}},\, u_2(x)=|\frac{1}{4}-x^2|^{\frac{1}{2}}$, which are both not locally Lipschitz. Applying mean squared error minimization for $d\in\{20,50\}$ to the exact moment vectors yields accurate polynomial approximations for both functions on their entire domains, even at the boundary and at the points where the functions are not locally Lipschitz, cf. Figure \ref{notLocLipFig}.
\end{ex}

\subsection{Recovering geometric objects}

One of the main applications of density estimation from moments is the shape reconstruction of geometric objects in image analysis. There has been extensive research on this topic, c.f. \cite{teague, liao} and the references therein. The reconstruction of geometric objects is a particular case of density estimation when $u=I_K$, i.e. the desired density $u$ is the indicator function of the geometric object $K\subset \Omega\subset\R^n$. Its given moments $y_{\alpha}=\int_{\Omega}\x^{\alpha}I_K(\x)d\lambda(\x)=\int_K \x^{\alpha}d\lambda(\x)$ do not depend on the frame $\Omega$. However, $\Omega$ does enter when computing the matrix $\M_d(\z)$ in problem (\ref{unconQP}). As indicated in Remark \ref{frameRmk} and demonstrated below, the choice of the enclosing frame $\Omega$ for $K$ is crucial for the pointwise approximation accuracy for a fixed order $d$. Since $u$ has a special structure, we derive an estimate $K_d$ for $K$ by choosing a superlevel set of $u^*_d$ as proposed in \cite{teague}:
\begin{equation*}
K_d := \{\x\in\R^n \mid u^*_d(\x)\geq \frac{1}{2} \}.
\end{equation*}

\begin{figure}[ht]
\begin{center}
\fbox{\includegraphics[width=0.18\textwidth, height=0.1\textheight]{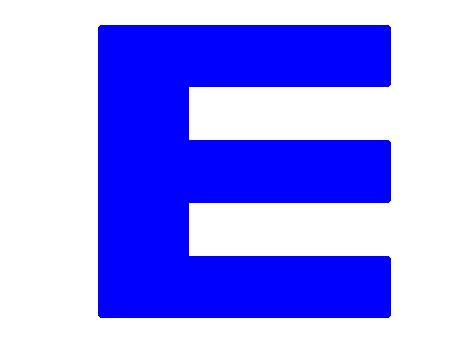}}
\fbox{\includegraphics[width=0.18\textwidth, height=0.1\textheight]{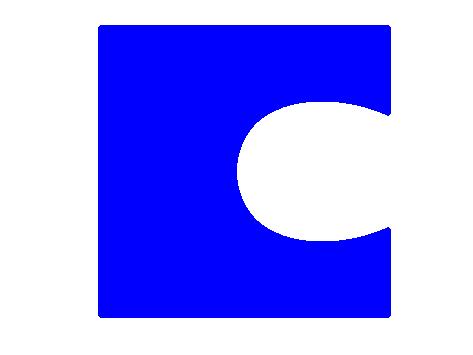}} 
\fbox{\includegraphics[width=0.18\textwidth, height=0.1\textheight]{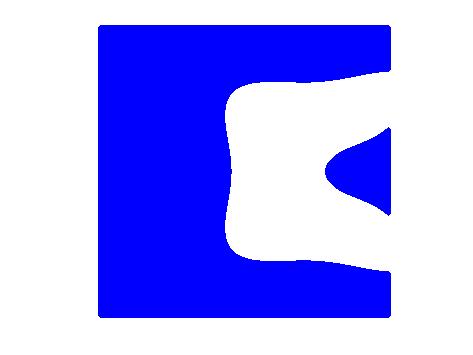}} 
\fbox{\includegraphics[width=0.18\textwidth, height=0.1\textheight]{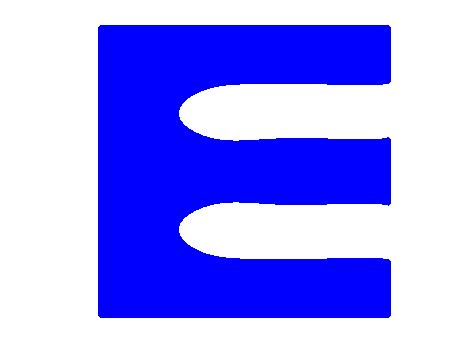}} 
\fbox{\includegraphics[width=0.18\textwidth, height=0.1\textheight]{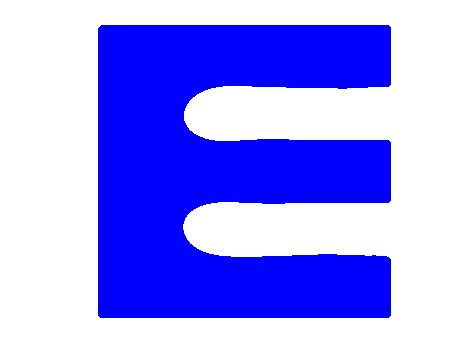}}\\
\fbox{\includegraphics[width=0.18\textwidth, height=0.1\textheight]{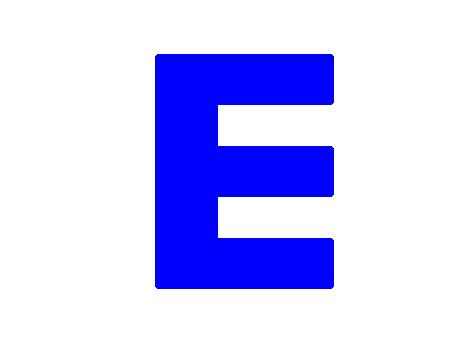}}
\fbox{\includegraphics[width=0.18\textwidth, height=0.1\textheight]{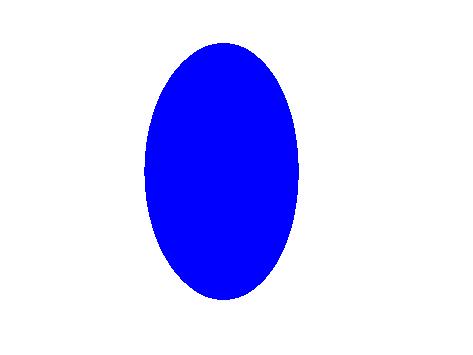}} 
\fbox{\includegraphics[width=0.18\textwidth, height=0.1\textheight]{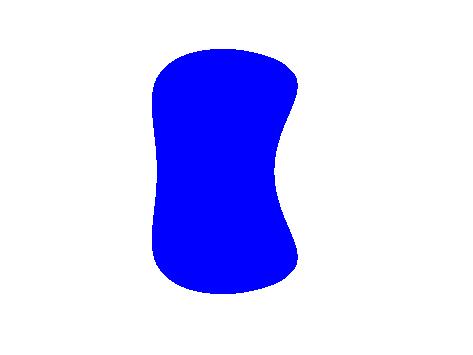}} 
\fbox{\includegraphics[width=0.18\textwidth, height=0.1\textheight]{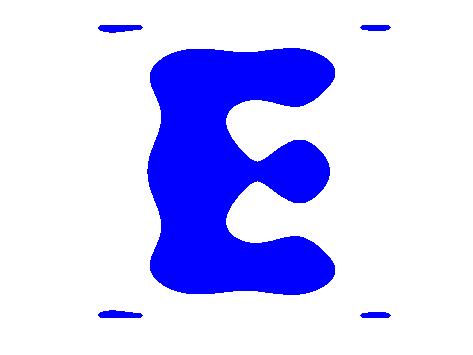}} 
\fbox{\includegraphics[width=0.18\textwidth, height=0.1\textheight]{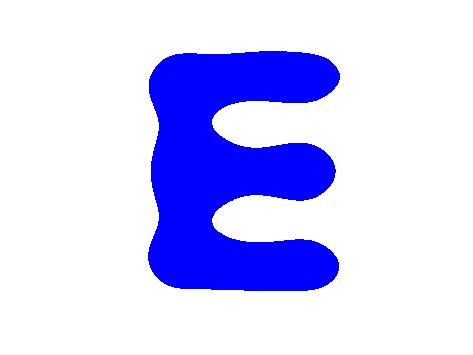}}\\
\fbox{\includegraphics[width=0.18\textwidth, height=0.1\textheight]{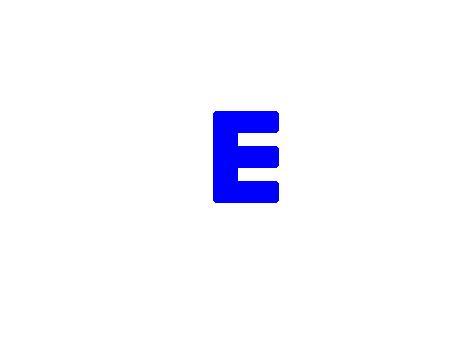}}
\fbox{\includegraphics[width=0.18\textwidth, height=0.1\textheight]{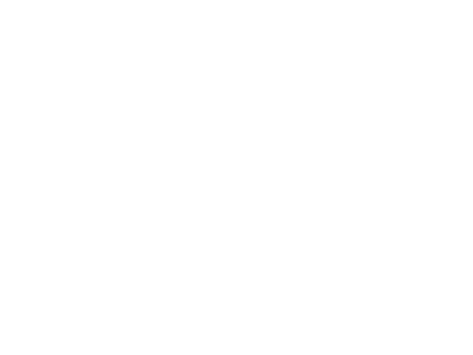}} 
\fbox{\includegraphics[width=0.18\textwidth, height=0.1\textheight]{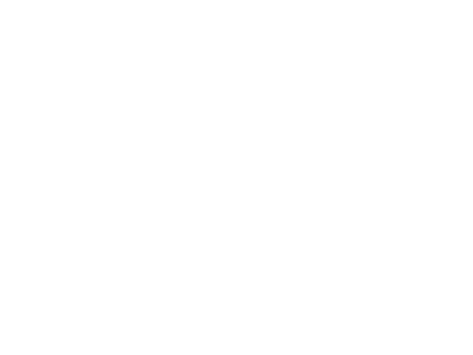}} 
\fbox{\includegraphics[width=0.18\textwidth, height=0.1\textheight]{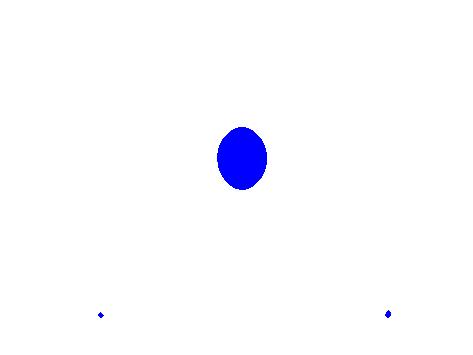}} 
\fbox{\includegraphics[width=0.18\textwidth, height=0.1\textheight]{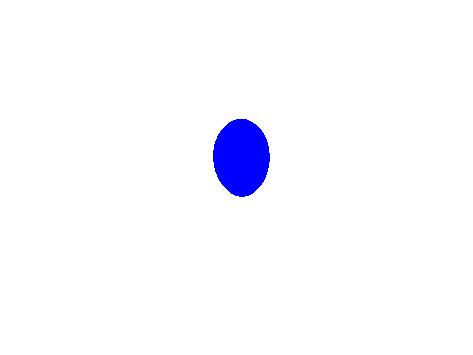}}
\caption{Recovering the letter E with mean squared error minimization; tight frames (top) vs. loose frames (bottom); original (left) vs. estimates for degrees $d\in\{3,5,8,10\}$ (from left to right).\label{EFig}}
\end{center}
\end{figure}

\begin{figure}[ht]
\begin{center}
\fbox{\includegraphics[width=0.18\textwidth, height=0.1\textheight]{Etight}}
\fbox{\includegraphics[width=0.18\textwidth, height=0.1\textheight]{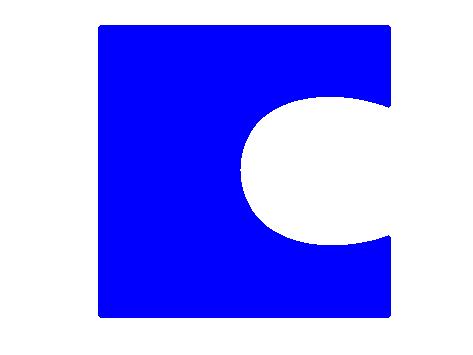}}
\fbox{\includegraphics[width=0.18\textwidth, height=0.1\textheight]{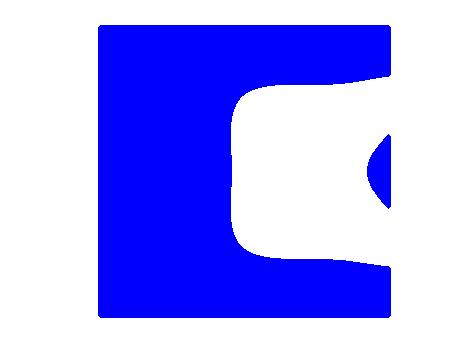}} 
\fbox{\includegraphics[width=0.18\textwidth, height=0.1\textheight]{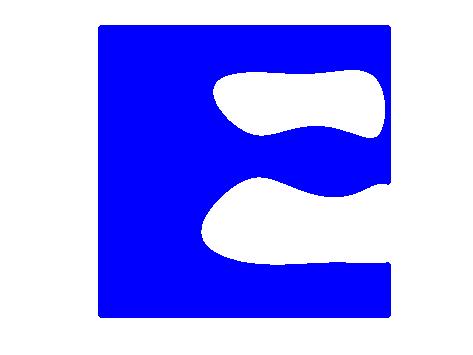}} 
\caption{Recovering the letter E with maximum entropy estimation; original (left) vs. estimates for degrees $d\in\{3,5,8\}$ (from left to right).\label{EFigMEE}}
\end{center}
\end{figure}

\begin{ex}
\label{exE}
A first object we recover is shaped like the letter E as in \cite{teague}. We determine the mean squared error minimizer for the same moment vector $\y$ with $d\in\{3,5,8,10\}$, but for three different frames $\Omega$. As pictured in Figure \ref{EFig}, we are able to reconstruct the estimates $K_d$ derived in \cite{teague} when $\Omega$ is chosen tight. Moreover, we observe that this good approximation of the severely nonconvex set $K$ and its discontinuous indicator function for a small number of moments $d$ depends heavily on the choice for $\Omega$. The wider the frame, the worse gets the approximation accuracy of the truncated estimate. Applying maximum entropy estimation for $d\in\{3,5,8\}$ yields density estimates of comparable accuracy than mean squared error minimization, cf. Figure \ref{EFigMEE}. However, the computational effort is much larger: for $d\in\{3,5,8\}$, the computational times for maximum entropy estimation are
82, 902 and 4108 seconds, respectively, whereas the mean squared error minimizer can be determined in less than one second for these values of $d$.
\end{ex}

\begin{figure}[ht]
\begin{center}
\fbox{\includegraphics[width=0.18\textwidth, height=0.1\textheight]{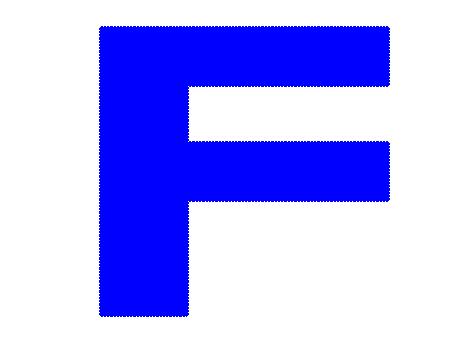}}
\fbox{\includegraphics[width=0.18\textwidth, height=0.1\textheight]{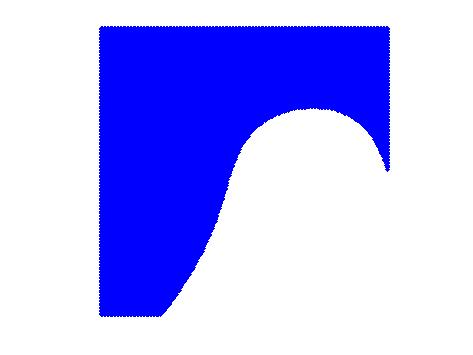}} 
\fbox{\includegraphics[width=0.18\textwidth, height=0.1\textheight]{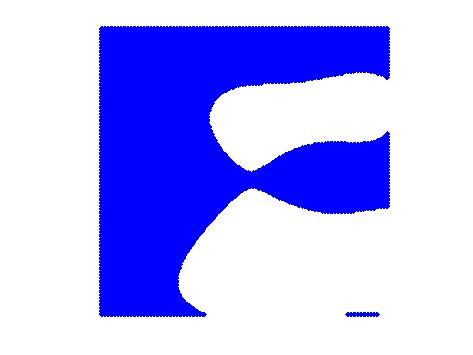}} 
\fbox{\includegraphics[width=0.18\textwidth, height=0.1\textheight]{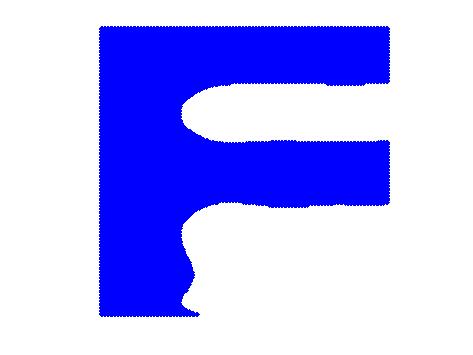}} 
\fbox{\includegraphics[width=0.18\textwidth, height=0.1\textheight]{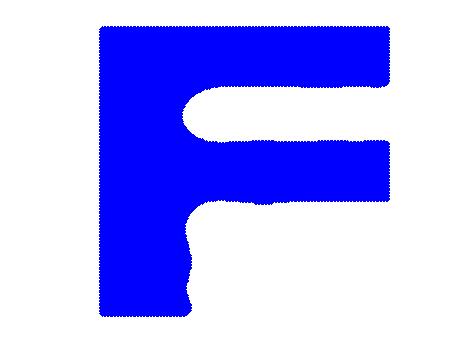}}\\
\fbox{\includegraphics[width=0.18\textwidth, height=0.1\textheight]{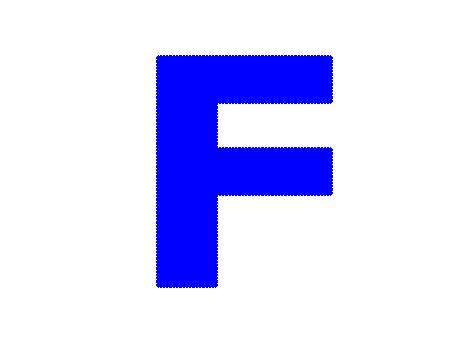}}
\fbox{\includegraphics[width=0.18\textwidth, height=0.1\textheight]{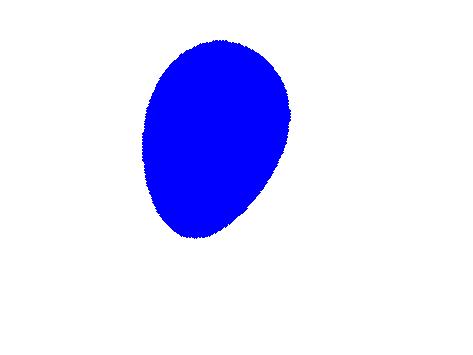}} 
\fbox{\includegraphics[width=0.18\textwidth, height=0.1\textheight]{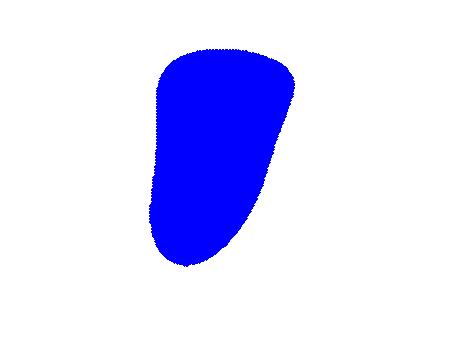}} 
\fbox{\includegraphics[width=0.18\textwidth, height=0.1\textheight]{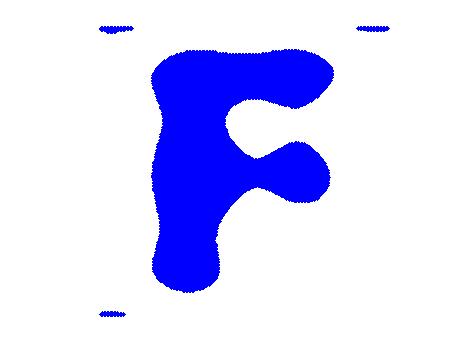}} 
\fbox{\includegraphics[width=0.18\textwidth, height=0.1\textheight]{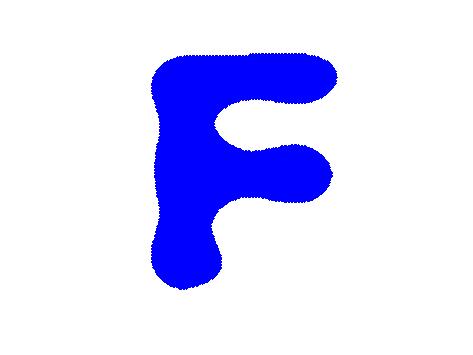}}\\
\fbox{\includegraphics[width=0.18\textwidth, height=0.1\textheight]{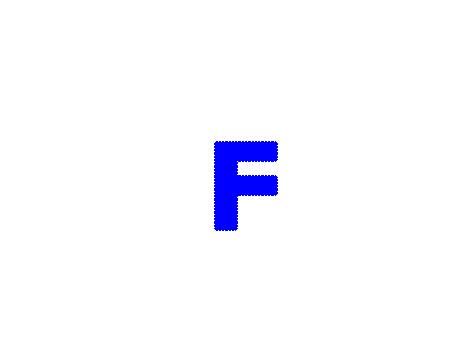}}
\fbox{\includegraphics[width=0.18\textwidth, height=0.1\textheight]{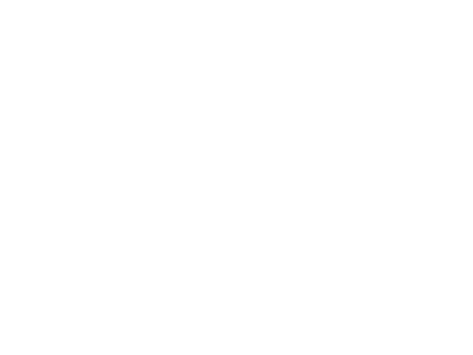}} 
\fbox{\includegraphics[width=0.18\textwidth, height=0.1\textheight]{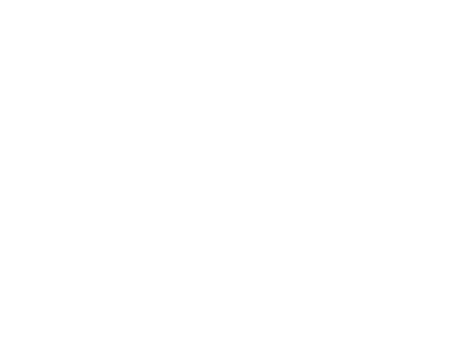}} 
\fbox{\includegraphics[width=0.18\textwidth, height=0.1\textheight]{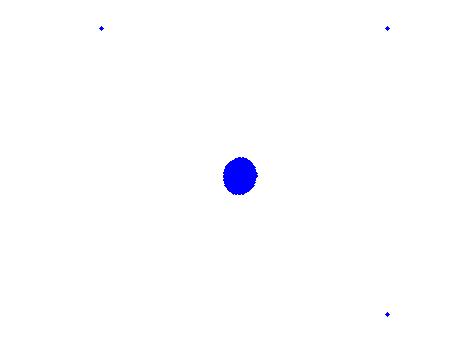}} 
\fbox{\includegraphics[width=0.18\textwidth, height=0.1\textheight]{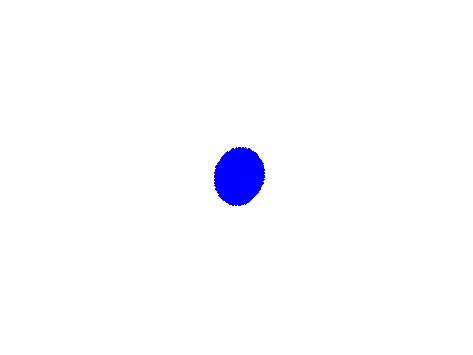}}
\caption{Recovering the letter F with mean squared error minimization; tight frames (top) vs. loose frames (bottom); original (left) vs. estimates for degrees $d\in\{3,5,8,10\}$ (from left to right).\label{FFig}}.
\end{center}
\end{figure}

\begin{ex}
Secondly we consider recovering an F-shaped set, a less symmetric example than the E-shaped set of Example \ref{exE}. As previously, we observe 
that the approximation accuracy of $K_d$ relies heavily on the frame $\Omega$. Even though the set $K$ has a complicated geometry, $K_{10}$ approximates $K$ accurately if $\Omega$ is chosen sufficiently tight, cf. Figure \ref{FFig}.
\end{ex}

\begin{ex}
Consider approximating $K:=\{\x\in\R^2\mid x_1(x_1^2-3x_2^2)+(x_1^2+x_2^2)^2\geq 0\}$, a nonconvex region enclosed by a trefoil curve.
Since this curve is of genus zero, its moment vector $\y$ can be determined exactly. Again, we need to choose an appropriate frame $\Omega\supset K$. The results for $\Omega=B(0,1)$ and $d\in\{3,5,8,10\}$ are pictured in Figure \ref{trefoilFig}.
\end{ex}

\begin{figure}[ht]
\begin{center}
\fbox{\includegraphics[width=0.18\textwidth, height=0.1\textheight]{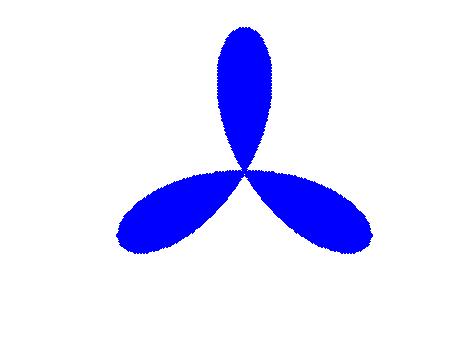}}
\fbox{\includegraphics[width=0.18\textwidth, height=0.1\textheight]{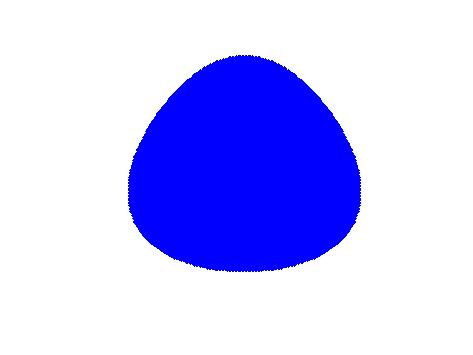}} 
\fbox{\includegraphics[width=0.18\textwidth, height=0.1\textheight]{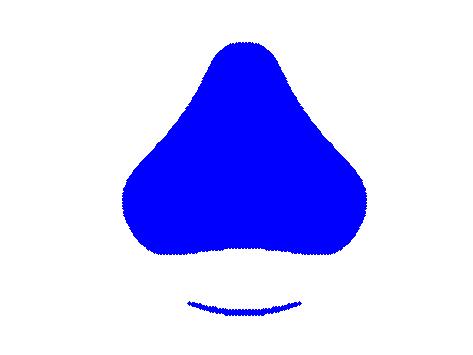}} 
\fbox{\includegraphics[width=0.18\textwidth, height=0.1\textheight]{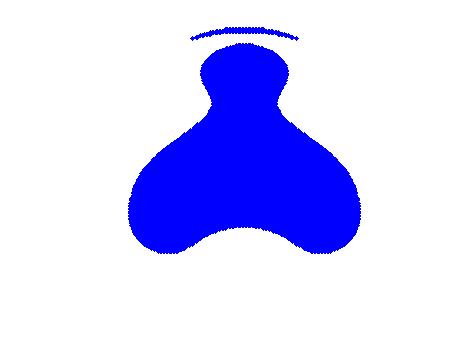}} 
\fbox{\includegraphics[width=0.18\textwidth, height=0.1\textheight]{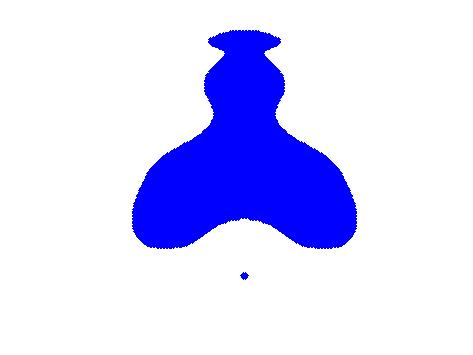}}
\caption{Trefoil region $K$ (left) and $\left((K_d\setminus K)\cup (K\setminus K_d) \right)$ for $d\in\{3,5,8,10\}$ (from left to right).\label{trefoilFig}}
\end{center}
\end{figure}

\subsection{Approximating solutions of differential equations}

Another important class of moment problems arises from the numerical analysis of ordinary and partial differential equations. Solutions of certain nonlinear differential equations can be understood as densities of measures associated with the particular differential equation. We obtain an approximate solution from the moment vector of this measure by solving an SDP problem. Approaches for deriving moment vectors associated with solutions of nonlinear differential equations have been introduced in \cite{mlhIfac, hlmOCMPDE} and are omitted here. We assume that the moment vector of a measure whose density is a solution of the respective differential equation is given in the following examples.\\

\begin{figure}[ht]
\begin{center}
\includegraphics[width=0.49\textwidth, height=0.2\textheight]{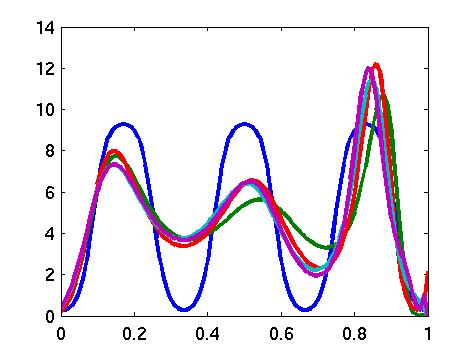}
\includegraphics[width=0.49\textwidth, height=0.2\textheight]{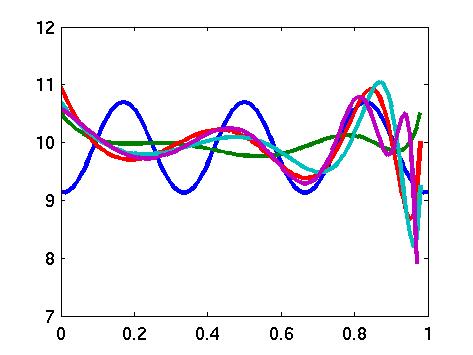}\\ 
\includegraphics[width=0.49\textwidth, height=0.2\textheight]{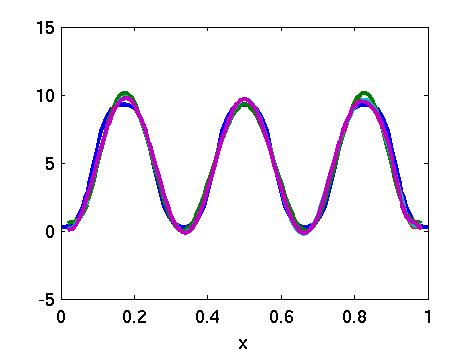}
\includegraphics[width=0.49\textwidth, height=0.2\textheight]{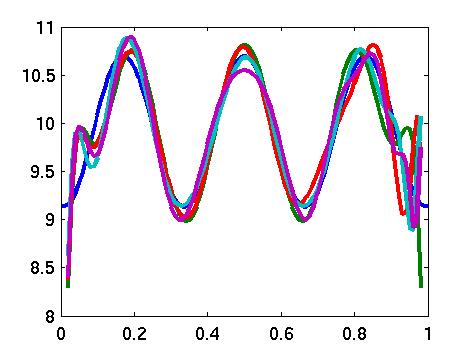} 
\caption{Estimates for $u$ (left) and $v$ (right) of the reaction-diffusion equation via maximum entropy estimation (above) and mean squared error minization (below).
Blue: $u$, green: $u^*_{10}$, red: $u^*_{20}$, cyan: $u^*_{30}$, magenta: $u^*_{50}$.\label{mimuraFig}}
\end{center}
\end{figure}

\begin{table}
\begin{center}
\begin{tabular}{|c|c|c|cc|cc|c|}
\hline $d$ & function & method & $\bar{\epsilon}_d(M)$ & $\hat{\epsilon}_d$ & $\bar{\epsilon}^o_d$ & $\hat{\epsilon}^o_d$ & time (sec.) \\
\hline 10 & $u$ & MEE & 2.0536 & 4.1063 & 2.1679 & 4.1063 & 38\\
30 & $u$ & MEE & 1.7066 & 3.4775 & 1.7974 & 3.4775 & 2 \\
50 & $u$ & MEE & 1.6978 & 3.3858 & 1.7792 & 3.3858 & 820 \\
\hline 10 & $u$ & $L^2$ & 0.4611 & 1.4506 & 0.4619 & 1.4506 & 0.1\\
30 & $u$ & $L^2$ & 0.3942 & 1.2431 & 0.4101 & 1.2431 & 0.1\\
50 & $u$ & $L^2$ & 0.3941 & 1.2173 & 0.4004 & 1.2173 & 0.1\\
\hline 10 & $v$ & MEE & 0.5782 & 2.0982 & 0.5205 & 1.0806 & 221\\
30 & $v$ & MEE & 0.5978 & 9.6060 & 0.4592 & 1.2356 & 645 \\
50 & $v$ & MEE & 0.6853 & 23.5993 & 0.4117 & 1.2973 & 3306 \\
\hline 10 & $v$ & $L^2$ & 0.3429 & 5.4767 & 0.1765 & 0.5617 & 0.1\\
30 & $v$ & $L^2$ & 0.3286 & 12.4501 & 0.1024 & 0.6253 & 0.1 \\
50 & $v$ & $L^2$ & 0.3744 & 14.2223 & 0.1454 & 0.5907 & 0.1 \\
\hline
\end{tabular}
\caption{Approximation accuracy of maximum entropy estimation (MEE) and mean squared error ($L^2$) optima for the reaction-diffusion equation.\label{mimuraResults}}
\end{center}
\end{table}

\begin{ex}
Given the moment vectors of a solution $(u,v)$ of the following reaction-diffusion equation \cite{mimura}:
\[
\begin{array}{ll}
\frac{1}{20} \; {u''}+ \frac{1}{9} \left( 35 + 16u - u^2 \right)\; u - u\, v = 0, \\
4 {v''} - \left( 1 + \frac{2}{5}v\right)\; v + u\, v = 0,\\
u'(0) = u'(5) = v'(0) = v'(5) = 0,\\
0\leq u,v  \leq 14 & \text{on }[0,5],
\end{array}
\]
we apply both maximum entropy estimation and mean squared error minimization to approximate the desired solution. For the numerical results, see Table \ref{mimuraResults} and Figure \ref{mimuraFig}. We observe that the mean squared error minimizers provide accurate pointwise approximations for $(u,v)$ on the entire domain, whereas the maximum entropy estimates provides a fairly accurate pointwise approximation on some segment of the domain only. Moreover, the mean squared error minimizers are obtained extremely fast as solutions of linear systems compared with the maximum entropy estimates. Thus, in this example, mean squared error minimization is clearly superior to maximum entropy when estimating densities from moments.
\end{ex}

\begin{table}
\begin{center}
\begin{tabular}{|c|c|c|c|c|}
\hline $d$ & method & $\bar{\epsilon}_d$ & $\bar{\epsilon}^o_d$ & time (sec.)\\
\hline 4 & MEE & 1.3e-2 & 1.1e-2 & 566\\
 6 & MEE & 9.7e-3 & 8.9e-3 & 2489 \\
\hline 4 & $L^2$ & 1.8e-3 & 1.4e-3 & 0.1\\
 6 & $L^2$ & 4.6e-4 & 2.8e-4 & 0.1\\
 10 & $L^2$ & 5.5e-4 & 2.1e-4 & 0.1\\
 12 & $L^2$ & 5.7e-4 & 2.1e-4 & 0.1\\
\hline
\end{tabular}
\caption{Approximation accuracy of maximum entropy estimation (MEE) and mean squared error ($L^2$) optima for the Allen-Cahn bifurcation PDE.\label{bifurResults}}
\end{center}
\end{table}

\begin{figure}[ht]
\begin{center}
\includegraphics[width=0.18\textwidth, height=0.2\textheight]{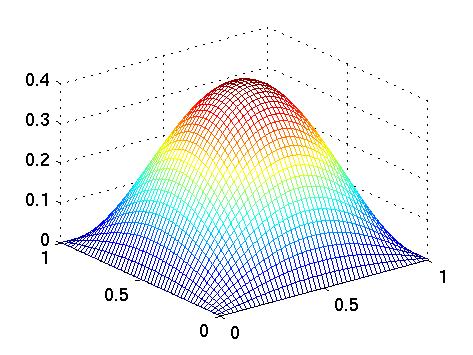}
\includegraphics[width=0.18\textwidth, height=0.2\textheight]{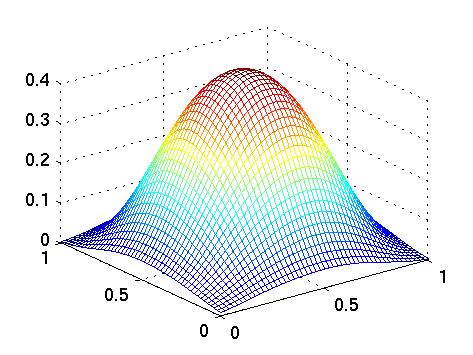}
\includegraphics[width=0.18\textwidth, height=0.2\textheight]{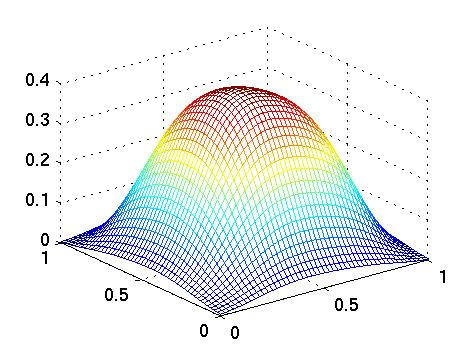}\\
\includegraphics[width=0.18\textwidth, height=0.2\textheight]{pdeBifurSol50}
\includegraphics[width=0.18\textwidth, height=0.2\textheight]{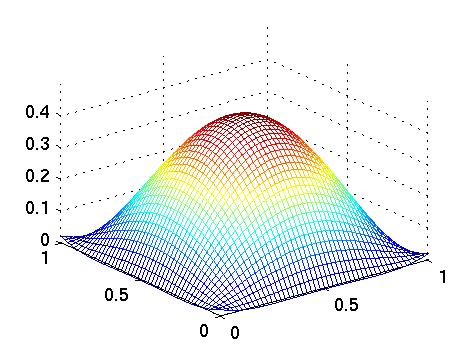}
\includegraphics[width=0.18\textwidth, height=0.2\textheight]{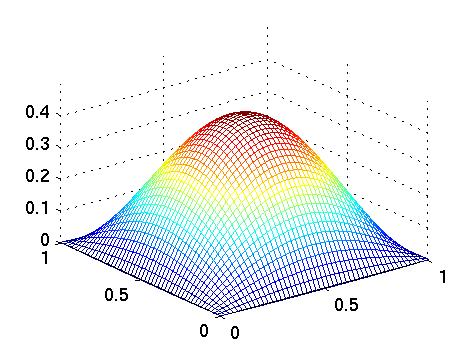}
\includegraphics[width=0.18\textwidth, height=0.2\textheight]{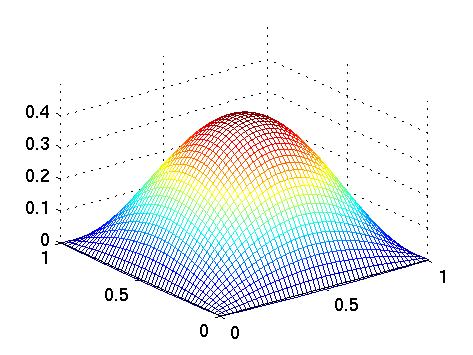}
\includegraphics[width=0.18\textwidth, height=0.2\textheight]{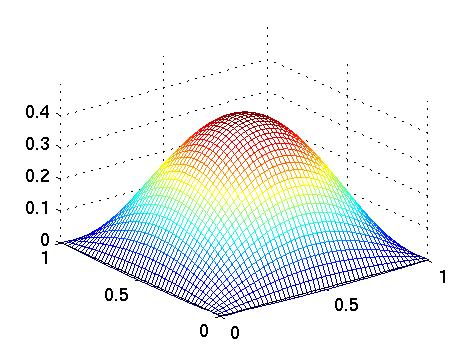}
\caption{Estimates for solution of the Allen-Cahn bifuraction PDE by maximum entropy estimation (above) and mean squared error minization (below). Exact solution (left) vs estimates for $d\in\{4,6\}$ and $d\in\{4,6,10,12\}$, respectively.\label{pdeBifurFig}}
\end{center}
\end{figure}

\begin{ex}
Given the moment vector of the nontrivial, positive solution of the Allen-Cahn bifurcation PDE:
\begin{equation}
\begin{array}{ll}
u_{xx} + u_{yy} + 22u(1-u^2) =0 & \text{on } [0,1]^2,\\
u = 0& \text{on } \partial [0,1]^2,\\
0\leq u\leq 1 & \text{on } [0,1]^2,
\end{array}
\label{pdeBifur}
\end{equation}
we apply both maximum entropy and mean squared error minization. The numerical results are reported in Table \ref{bifurResults}. The approximation accuracy for both methods is comparable, with mean squared error minimization being slightly more precise. However, applying maximum entropy estimation is limited for this problem as the cases $d>6$ are numerically too heavy to be solved in reasonable time. Mean squared error minimization yields increasingly better estimates for the desired solution within seconds, as pictured in Figure \ref{pdeBifurFig}. 
\end{ex}

\begin{table}
\begin{center}
\begin{tabular}{|c|c|c|c|c|}
\hline $d$ & method & $\bar{\epsilon}_d$ & $\bar{\epsilon}^o_d$ & time (sec.)\\
\hline 3 & MEE & 5.4e-2 & 5.1e-2 & 20\\
 6 & MEE & 1.9e-2 & 1.9e-2 & 2192\\
\hline 3 & $L^2$ & 2.8e-2 & 2.6e-2 & 0.1\\
 6 & $L^2$ & 2.8e-2 & 2.6e-2 & 0.1\\
 10 & $L^2$ & 1.4e-2 & 1.4e-2 & 0.2 \\
\hline
\end{tabular}
\caption{Approximation accuracy of maximum entropy estimation (MEE) and mean squared error ($L^2$) minization for the eikonal PDE.\label{eikonalResults}}
\end{center}
\end{table}

\begin{figure}[ht]
\begin{center}
\includegraphics[width=0.32\textwidth, height=0.2\textheight]{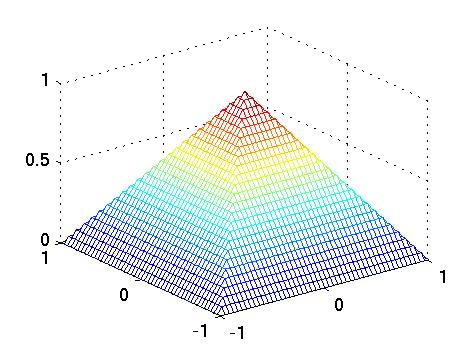}
\includegraphics[width=0.32\textwidth, height=0.2\textheight]{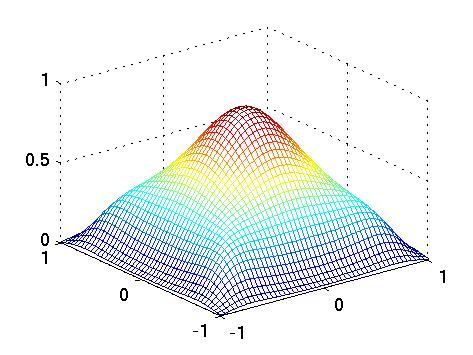}
\includegraphics[width=0.32\textwidth, height=0.2\textheight]{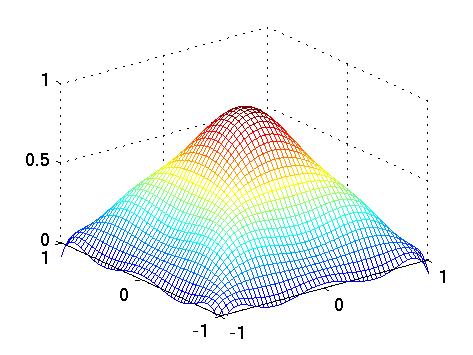}
\caption{Estimates for solution (left) of eikonal PDE by maximum entropy estimation (center) for $d=6$ and mean squared error minimization (right) for $d=10$.\label{eikonalFig}}
\end{center}
\end{figure}

\begin{ex}
Given the moment vector of the viscosity solution of the classical eikonal PDE:
\begin{equation}
\begin{array}{ll}
u_{x}^2 + u_{y}^2 -1 =0 & \text{on } [0,1]^2,\\
u = 0 & \text{on } \partial [0,1]^2,
\end{array}
\label{pdeEikonal}
\end{equation}
we apply both maximum entropy and mean squared error minization. Both methods provide better pointwise approximates for increasing degrees $d$, cf. Table \ref{eikonalResults} and Figure \ref{eikonalFig}. Again, the advantage of using mean squared error minimization is its negligible computation time.
\end{ex}

\subsection{Approximating indicator functions}
In all examples discussed so far, the polynomial approximates $u^*_d$ of solution $u$ obtained by solving the unconstrained problem (\ref{unconOpt}) have been nonnegative on $\Omega$. It was therefore not necessary to solve the constrained problems (\ref{noliftOpt}) or (\ref{putinarOpt}). 

\begin{table}
\begin{center}
\begin{tabular}{|c|c|c|c|c|}
\hline $d$ & problem & $\bar{\epsilon}_d$ & $\hat{\epsilon}_d$ & time (sec.) \\
\hline 10  & (\ref{unconOpt})& 0.08 & 0.50 & 0.1\\
 10 & (\ref{putinarOpt}) & 0.11 & 0.56 &  0.9 \\
\hline 50 & (\ref{unconOpt}) & 0.05 & 0.50 &  0.1\\
 50 & (\ref{putinarOpt}) &  0.08 & 0.54 &  2.7 \\
\hline 100 & (\ref{unconOpt}) & 0.05 & 0.50 &  0.3 \\
 100 & (\ref{putinarOpt}) & 0.07 & 0.55 & 37.5 \\
\hline
\end{tabular}
\caption{Mean squared error minimization for $u(x)=I_{[0.5, 1]}(x)$.\label{indFunPutinarRes}}
\end{center}
\end{table}

\begin{figure}[ht]
\begin{center}
\includegraphics[width=0.6\textwidth, height=0.35\textheight]{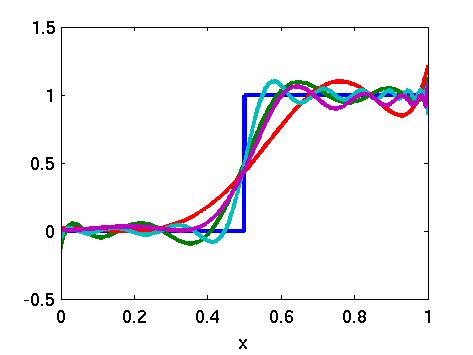}
\caption{Unconstrained and constrained mean squared error minimization estimates.
Blue: $u$, green: $u^*_{10}$ unconstrained, red: $u^*_{10}$ constrained, cyan: $u^*_{100}$ unconstrained,
magenta: $u^*_{100}$ constrained.\label{indFunPutinarFig}}
\end{center}
\end{figure}

\begin{ex}
\label{indFunPutinarEx}
As a first simple example where solving one of the constrained problems is required, consider $u:\, [0,1]\rightarrow\R,u(x)=I_{[0.5, 1]}(x)$ and its given vector of moments. We solve problems (\ref{unconOpt}) and (\ref{putinarOpt}) for $d\in\{10,50,100\}$. As illustrated in Figure \ref{indFunPutinarFig}, solving problem (\ref{putinarOpt}) provides globally nonnegative estimates for $u$. This comes at the price of decreasing approximation accuracy compared to solving the unconstrained problem, since the feasible set of problem (\ref{putinarOpt}) is a subset of the feasible set of (\ref{unconOpt}), cf. Table \ref{indFunPutinarRes}.
\end{ex}

\section{Orthogonal bases and the Gibbs effect}
\label{secGibbs}

Density estimation via $L^2$ norm minimization over the truncated space of polynomials is a special case of approximating a function belonging to a complicated functional space by a linear combination of basis elements of an easier, well-understood functional space. This general setting has been considered for both real and complex functional spaces. In the following we discuss the particular case when the basis of the easier functional space is orthogonal, and also effects resulting from truncating the infinite dimensional series expansion of the unknown function.

\subsection{The complex case}
\label{subsecComplex}

The classic, complex analogue of the real, inverse moment problem discussed in the previous sections is the problem of approximating a periodic function $u$ by a trigonometric polynomial. An orthonormal basis for the space of trigonometric polynomials is given by $(e^{-ikx})_{k \in {\mathbb Z}}$. It provides the {Fourier series} expansion:
\begin{equation*}
u(x) = \sum_{k\in{\mathbb Z}} y_k e^{ikx},
\end{equation*}
where $y_k=\frac{1}{2\pi}\int_{-\pi}^{\pi}u(x)e^{-ikx}dx$ in the univariate case, or
\begin{equation*}
u(x,y) = \sum_{k,l\in{\mathbb Z}} y_{k,l} e^{ikx}e^{ily},
\end{equation*}
where $y_{k,l}=\frac{1}{4\pi^2}\int_{-\pi}^{\pi}\int_{-\pi}^{\pi}u(x,y)e^{-ikx}e^{-ily}dx\, dy$ in the bivariate case. One obtains a trigonometric polynomial approximation for $u$ when truncating this series expansion at some $d\in\N$,
\begin{equation*}
u_d(x) = \sum_{k=-d}^d y_k e^{ikx}.
\end{equation*}
It is a well-known fact in Fourier theory that
\begin{equation*}
\int_0^{2\pi}(u-u_d)^2dx \rightarrow 0 \text{  for } d\rightarrow\infty.
\end{equation*}
The Fourier approximation for a periodic function is therefore the trigonometric analogue of the real polynomial approximation for a function obtained by mean squared error minimization in our approach. As in the real case we have almost uniform convergence. However, $u_d$ does not converge to $u$ uniformly if $u$ piecewise continuously differentiable with jump discountinuities, due to an effect known as the Gibbs phenomenon. It states that the truncated approximation shows a near constant overshoot and undershoot near a jump discontinuity. This overshoot does not vanish, it only moves closer to the jump for increasing $d$.

Since the truncated Fourier series is the trigonometric analogue of unconstrained mean squared norm minimizer of problem (\ref{unconOpt}), the question arises of how the trigonometric estimate behaves for a periodic functions when adding nonnegativity constraints as in problems (\ref{noliftOpt}) and (\ref{putinarOpt}) which aim at preventing the estimate from over- or undershooting near jump discontinuities.

In order to derive a tractable SDP problem, the nonnegativity constraints for the trigonometric polynomial need to be relaxed or tightened to LMI constraints. The difference with the real case will be the moment and localizing matrices being of Toeplitz type in contrast with the Hankel type matrices in the constraints of problem (\ref{noliftOpt}).

\subsection{The real case}

In our discussions in the previous sections we approximated an unknown density by a linear combination of elements of the monomial basis of the space of real polynomials. This choice of a basis for $\R [x]$ has several theoretical and practical shortcomings. For once, it is not an orthogonal basis with respect to the Lebesgue measure, and moreover the moment matrix $\M_d(\z)$ in problem (\ref{unconOpt}), whose inverse needs to be computed to determine $u^*_d$, is severely ill-conditioned \cite{bertero,teague}. As already pointed out in \cite{teague}, when choosing an orthogonal basis such as Legendre or Chebychev polynomials for $\R [x]$, the $L^2$ norm minimizer $u^*_d$ can be determined according to a closed-form formula. Thus, we do not even have to solve a linear system of equations.
This is essentially the real analogue of the closed form formula for the coefficients in the Fourier series approximation of periodic functions.

In \S \ref{subsecComplex} we discussed the Gibbs phenomenon. For reasons outlined there, we expect to observe this effect in the real case as well, when approximating an $L^2$ integrable function with jump discontinuities. In fact, the function from Example \ref{indFunPutinarEx} illustrates that. As shown in Figure \ref{indFunPutinarFig}, we observe an over- and undershoot on both sides of the discontinuity. The amplitude of this overshoot does not decrease for increasing $d$, but the overshoot moves closer to the jump. When adding the global nonnegativity constraint for the polynomial estimate, the undershoot at the left side disappears. However, it is compensated for by a weaker, overall pointwise approximation accuracy of the estimate. This observation is a first, partial answer to the question raised in the complex case.

\section{Conclusion}
\label{secConclusion}

We introduced an approach for estimating the density of a measure given a finite number of its moments only,
in the multivarite case and with no continuity assumption on the density. As an estimate we choose the polynomial minimizing the $L^2$ norm distance, or mean squared error to the unknown density. We have shown that this estimate is easy to determine by solving a linear system of equations. Moreover, it converges almost uniformly towards the desired density when the degree
increases. By minimizing the mean squared error subject to additional linear matrix inequality constraints, which translates to solving a semidefinite program, we obtain a density estimate guaranteed to be nonnegative on the support of the measure. Also in the constrained case, we have shown almost uniform convergence of the nonnegative mean squared error minimizer towards the unknown density, when the degree increases.
Mean squared error minimization is often superior to maximum entropy estimation in terms of approximation accuracy and computation time, as demonstrated for a number of examples. Moreover, it is able to handle general, basic, compact semialgebraic sets as support for the unknown measure, which present a challenge to maximum entropy estimation in higher dimension.



\begin{thebibliography}{las}
\bibitem{ash}
R.B. Ash. {\it Real Analysis and Probability}, Academic Press Inc., Boston, 1972.
\bibitem{athanassoulis}
G.A. Athanassoulis, P.N. Gavriliadis. The truncated Hausdorff moment problem solved by using kernel density functions, Probabilistic Engineering Mechanics 17:273-291, 2002.
\bibitem{pade}
G. A. Baker, P. Graves-Morris. Pad\'e approximants. Addison-Wesley, Reading, MA, 1981.
\bibitem{bertero}
M. Bertero, C. De Mol, E. R. Pike. Linear inverse problems with discrete data. I: General formulation and singular system analysis. Inverse Problems 1(4):301--330, 1985.
\bibitem{borwein}
J. M. Borwein, A. S. Lewis. On the convergence of moment problems. Trans. Amer. Math. Soc. 325(1):249--271, 1991.
\bibitem{donoho}
D. L. Donoho, I. M. Johnstone, G. Kerkyacharian, D. Picard. Density estimation by wavelet thresholding. Ann. Statist. 24(2):508--539, 1996.
\bibitem{eloyan}
A. Eloyan, S. K. Ghosh. Smooth density estimation with moment constraints using mixture distributions. J. Nonparamet. Stat. 23(2):513--531, 2011.
\bibitem{gavriliadis}
P.N. Gavriliadis, G.A. Athanassoulis. The truncated Stieltjes moment problem solved by using kernel density functions, Journal of Computational and Applied Mathematics 236:4193-4213, 2012.
\bibitem{goodrich}
R. K. Goodrich, A. Steinhardt. $L_2$ spectral estimation, SIAM J. Appl. Math. 46(3):417--426, 1986.
\bibitem{hlmOCMPDE}
D. Henrion, J. B. Lasserre, M. Mevissen. The generalized problem of moments for nonlinear differential equations.
Work in progress, 2012.
\bibitem{izenman}
A. J. Izenman. Recent developments in nonparametric density estimation. J. Amer. Stat. Assoc. 86(413):205--224, 1991.
\bibitem{jaynes1}
E. T. Jaynes. Information theory and statistical mechanics, Phys. Rev. Series II 106(4):620--630, 1957.
\bibitem{jaynes2}
E. T. Jaynes. Information theory and statistical mechanics II. Phys. Rev. Series II 108(2):171--190, 1957. 
\bibitem{john}
V. John, I. Angelov, A. A. \"Onc\"ul, D. Th\'evenin. Techniques for the reconstruction of a distribution from a finite number of its moments. Chemical Engineering Science 62(11):2890--2904, 2007.
\bibitem{jones}
L. K. Jones, V. Trutzer, On extending the orthogonality property of minimum norm solutions in Hilbert space to general methods for linear inverse problems, Inverse Problems 6(3):379--388, 1990.
\bibitem{kerkyacharian}
G. Kerkyacharian, D. Picard. Density estimation by kernel and wavelet methods: optimality of Besov spaces, Stat. Probab. Letters 18(4):327--336, 1993.
\bibitem{landau}
H. J. Landau. Moments in Mathematics.
Proc. Symp. Appl. Math., Vol. 37, Amer. Math. Soc., 1987.
\bibitem{lasserre}
J. B. Lasserre. Moments, positive polynomials and their applications.
Imperial College Press, London, UK, 2009.
\bibitem{newlook}
J. B. Lasserre. A new look at nonnegativity on closed sets and polynomial optimization.
SIAM J. Optim. 21:864--885, 2011.
\bibitem{liao}
S. X. Liao, M. Pawlak. On image analysis by moments. IEEE Trans. Pattern Analysis and Machine Intelligence 18(3):254--266, 1996.
\bibitem{mead}
L. R. Mead, N. Papanicolaou. Maximum entropy in the problem of moments. J. Math. Phys. 25(8):2404--2417, 1984.
\bibitem{mlhIfac}
M. Mevissen, J. B. Lasserre, D. Henrion. Moment and SDP relaxation techniques for smooth approximations of problems involving nonlinear differential equation.
Proc. IFAC World Congress on Automatic Control, Milan, Italy, 2011.
\bibitem{mimura}
M. Mimura. Asymptotic behaviors of a parabolic system related to a planktonic prey and predator model.
SIAM J. Appl. Math. 37(3):499-512, 1979.
\bibitem{mnatsa}
R. M. Mnatsakanov. Moment-recovered approximations of multivariate distributions: the Laplace transform inversion.
Stat. Prob. Letters 81(1):1--7, 2011.
\bibitem{parzen}
E. Parzen. On estimation of a probability density function and mode. Ann. Math. Stat. 33:1065-1076, 1962.
\bibitem{putinar} 
M. Putinar. Positive polynomials on compact semi-algebraic sets. Indiana Univ. Math. J. 42:969--984, 1993.
\bibitem{provost}
S.B. Provost. Moment-Based Density Approximations, The Mathematica Journal 9(4):727-756, 2005.
\bibitem{talenti}
G. Talenti. Recovering a function from a finite number of moments. Inverse Problems 3(3):501-517, 1987.
\bibitem{teague}
M. R. Teague. Image analysis via the general theory of moments. J. Opt. Soc. Amer. 70(8):920-930, 1980.
\bibitem{vannucci}
M. Vannucci. Nonparametric density estimation using wavelets. Discussion Paper 95-26, ISDS, Duke University, 1998.
\end{thebibliography}
\end{document}